\documentclass[12pt,oneside]{amsart}
\usepackage{amscd,amsfonts,amsmath,amsthm,amssymb,verbatim,mathrsfs}
\usepackage[dvips]{graphicx}
\usepackage{tikz}
\usetikzlibrary{shapes,backgrounds}
\usetikzlibrary{calc}
\usetikzlibrary{matrix}
\usepackage{graphics,hyperref}
\usepackage{enumitem}
\usepackage{psfrag}
\usepackage{pstcol,pst-plot,pst-3d}
\usepackage{latexsym}
\usepackage{pstcol,pst-plot,pst-3d}
\usepackage{multicol}
\usepackage[normalem]{ulem} 

\psset{unit=0.7cm,linewidth=0.8pt,arrowsize=2.5pt 4}

\newpsstyle{fatline}{linewidth=1.5pt}
\newpsstyle{fyp}{fillstyle=solid,fillcolor=verylight}
\definecolor{verylight}{gray}{0.97}
\definecolor{light}{gray}{0.9}
\definecolor{medium}{gray}{0.85}

\def\NZQ{\Bbb}               
\def\NN{{\NZQ N}}

\def\ZZ{{\NZQ Z}}

\def\frk{\frak}               

\def\Phi{{\frk n}}
\def\Phi{{\frk N}}
\def\MC{{\mathcal C}}

\def\ME{{\mathcal E}}

\def\MH{{\mathcal H}}
\def\MS{{\mathcal S}}
\def\MU{{\mathcal U}}

\def\Gr{{\mathcal Gr}}
\def\ub{{\bold u}}
\def\vb{{\bold v}}
\def\ab{{\bold a}}

\def\cb{{\bold c}}

\def\opn#1#2{\def#1{\operatorname{#2}}} 
\opn\chara{char} \opn\length{\ell} \opn\pd{pd} \opn\rk{rk}
\opn\projdim{proj\,dim} \opn\injdim{inj\,dim} \opn\rank{rank}
\opn\depth{depth} \opn\grade{grade} \opn\height{height}
\opn\embdim{emb\,dim} \opn\codim{codim} \opn\sgn{sgn}

\opn\Tr{Tr} \opn\bigrank{big\,rank}
\opn\superheight{superheight}\opn\lcm{lcm}
\opn\trdeg{tr\,deg}
\opn\reg{reg} \opn\lreg{lreg} \opn\ini{in} \opn\lpd{lpd}
\opn\size{size}\opn\bigsize{bigsize}
\opn\cosize{cosize}\opn\bigcosize{bigcosize}
\opn\sdepth{sdepth}\opn\sreg{sreg}
\opn\link{link}\opn\fdepth{fdepth} \opn\generic{generic}
\opn\div{div} \opn\Div{Div} \opn\cl{cl} \opn\Cl{Cl} \opn\Cor{Cor}
\opn\Spec{Spec} \opn\Supp{Supp} \opn\supp{supp} \opn\Sing{Sing}
\opn\Ass{Ass} \opn\Min{Min}\opn\Mon{Mon} \opn\dstab{dstab} \opn\astab{astab}
\opn\Ann{Ann} \opn\Rad{Rad} \opn\Soc{Soc} 
\opn\Im{Im} \opn\Ker{Ker} \opn\Coker{Coker} \opn\Am{Am}
\opn\Hom{Hom} \opn\Tor{Tor} \opn\Ext{Ext} \opn\End{End}
\opn\Aut{Aut} \opn\id{id} \opn\span{span}

\opn\nat{nat}
\opn\pff{pf}
\opn\Pf{Pf} \opn\GL{GL} \opn\SL{SL} \opn\mod{mod} \opn\ord{ord}
\opn\Gin{Gin} \opn\Hilb{Hilb}\opn\sort{sort}
\opn\aff{aff} \opn\con{conv} \opn\relint{relint} \opn\st{st}
\opn\lk{lk} \opn\cn{cn} \opn\core{core} \opn\vol{vol}
\opn\link{link} \opn\star{star}\opn\lex{lex} 
\opn\gr{gr}

\def\pot#1#2{#1[\kern-0.28ex[#2]\kern-0.28ex]}

\opn\dirlim{\underrightarrow{\lim}}
\opn\inivlim{\underleftarrow{\lim}}
\def\Implies{\ifmmode\Longrightarrow \else
        \unskip${}\Longrightarrow{}$\ignorespaces\fi}
\def\implies{\ifmmode\Rightarrow \else
        \unskip${}\Rightarrow{}$\ignorespaces\fi}
\def\iff{\ifmmode\Longleftrightarrow \else
        \unskip${}\Longleftrightarrow{}$\ignorespaces\fi}

\let\:=\colon
\newtheorem{Theorem}{Theorem}[section]

\newtheorem{Corollary}[Theorem]{Corollary}
\newtheorem{Proposition}[Theorem]{Proposition}
\newtheorem{Remark}[Theorem]{Remark}

\newtheorem{Example}[Theorem]{Example}

\newtheorem{Definition}[Theorem]{Definition}

\let\epsilon\varepsilon
\let\kappa=\varkappa
\makeatletter

\def\qed{\ifhmode\textqed\fi
      \ifmmode\ifinner\quad\qedsymbol\else\dispqed\fi\fi}
\def\textqed{\unskip\nobreak\penalty50
       \hskip2em\hbox{}\nobreak\hfil\qedsymbol
       \parfillskip=0pt \finalhyphendemerits=0}
\def\dispqed{\rlap{\qquad\qedsymbol}}

\opn\dis{dis}
\def\pnt{{\raise0.5mm\hbox{\large\bf.}}}

\opn\Lex{Lex}

\begin{document}

\title{ Hypergraph encodings of arbitrary toric ideals }

\author{Sonja Petrovi\'c, Apostolos Thoma, Marius Vladoiu}

\address{Sonja Petrovi\'c, Department of Applied Mathematics,  Illinois Institute of Technology, Chicago 60616, USA}
\email{sonja.petrovic@iit.edu}

\address{Apostolos Thoma, Department of Mathematics, University of Ioannina, Ioannina 45110, Greece} 
\email{athoma@uoi.gr}

\address{Marius Vladoiu, Faculty of Mathematics and Computer Science, University of Bucharest, Str. Academiei 14, Bucharest, RO-010014, Romania, and}
\address{Simion Stoilow Institute of Mathematics of Romanian Academy, Research group of the project PN-II-RU-TE-2012-3-0161, P.O.Box 1--764, Bucharest
014700, Romania}
\email{vladoiu@fmi.unibuc.ro}


\begin{abstract}
Relying on the combinatorial classification of toric ideals using their bouquet structure, we focus on toric ideals of hypergraphs 
and study how they relate to general toric ideals. 
We show that hypergraphs exhibit a surprisingly general behavior: the toric ideal associated to any general matrix  can be encoded by that of a $0/1$ matrix, while preserving the essential combinatorics of the original ideal. We provide two universality results about the unboundedness of degrees of various generating sets: minimal, Graver, universal Gr\"obner bases, and indispensable binomials. Finally, we provide a polarization-type operation for arbitrary positively graded toric ideals, which preserves all the combinatorial signatures and the homological properties of the original toric ideal. 
\end{abstract}
\maketitle

\section*{Introduction} 
Toric ideals associated to $0/1$ matrices  occupy a special place in the world of toric ideals due to their applications to biology \cite{SS}, algebraic statistics \cite{SaSt}, integer programming \cite{DHOW}, matroid theory \cite{LM}, combinatorics \cite{PeSt}. 
They are, by definition, incidence matrices of hypergraphs, and in the special case when there are only two nonzero entries in each column they are the incidence matrices  of graphs. There is an abundant literature for toric ideals of graphs, including \cite{OH1,RTT,TT,Vi}, and the book \cite{Vill}, while toric ideals of hypergraphs were recently studied in \cite{GP} and \cite{PeSt}. 

Over the last decades, the rich combinatorial structure of toric ideals has been an anchor for many investigations of their properties.
 With this literature in mind in \cite{PTV} we defined and studied the \emph{bouquet ideal} of an integer matrix $A$ and its relationship to the combinatorial and algebraic structure of the toric ideal 
 encoded by $A$. Bouquets of a matrix naturally arise from an oriented matroid that records the dependencies from the Gale transform of $A$.  
Specifically, in \cite{PTV} we showed that the bouquets of a matrix capture the essential combinatorics of the corresponding toric ideal, leading to a classification of toric ideals with respect  to the bouquet structure. In addition, the bouquet construction represented the key step for a combinatorial classification of strongly robust toric ideals and provides a technique for constructing infinitely many examples of the prominent classes of toric ideals with rich combinatorics, due to equality of various bases, that include unimodular (\cite{BPS}), generic (\cite{PS}) and robust ideals (\cite{BR}). 

In light of these results, here we focus on toric ideals of $0/1$ matrices and study their relationship to general toric ideals. Similarly to how the usual reference to  (Markov, Gr\"obner, Graver) bases of matrices refers to bases of corresponding toric ideals, throughout the paper we refer to bases of hypergraphs when we think of bases of toric ideals of incidence matrices of those hypergraphs. We refer the reader to \cite{St,PTV} for general definitions of the above-mentioned bases together with the circuits and indispensable binomials, to \cite{St,CTV} for some equivalent algebraic descriptions, and to \cite{St,CTV2} for the general hierarchy among them. As in \cite{PTV}, we will denote the Graver basis of the matrix $A$ by $\Gr(A)$, its set of circuits by $\MC(A)$, and we recall that since any element of $\Ker_{\ZZ}(A)$ can be written as a (conformal) sum of elements from $\Gr(A)$ then one can say that the essential combinatorial information of the toric ideal of $A$ is provided by $\Gr(A)$. For the remainder of the introduction, we summarize and put into perspective our two main results. 

The first main result says, roughly speaking, that  the essential combinatorics of \emph{arbitrary} toric ideals is encoded by  almost $3$-uniform hypergraphs, that is,  hypergraphs with edges of cardinality at most $3$ and having at least one such edge:

\smallskip
\noindent\textbf{Theorem 1} [Theorem \ref{worse}: Hypergraph encodings of arbitrary toric ideals].
\textit{Given any integer matrix $A$, there exists an almost $3$-uniform hypergraph $\MH=(V,\ME)$ such that the toric ideal of the hypergraph $I_{\MH}$ is strongly robust and there is a bijective correspondence between the elements of $\Ker_{\ZZ}(A)$, $\Gr(A)$, and $\mathcal C(A)$ and  $\Ker_{\ZZ}(\MH)$, $\Gr(\MH)$, and $\mathcal C(\MH)$, respectively.}
\smallskip

Thus toric ideals of almost $3$-uniform hypergraphs, that is toric ideals of $0/1$-matrices with at most three 1's on each column are ``as complicated'' as \emph{any arbitrary} toric ideals with respect to its essential combinatorics, the Graver basis. Recall that if $H$ is a $0/1$ matrix, the toric ideal $I_{H}$ is positively graded, that is, $\Ker_{\ZZ}(H)\cap\NN^n=\{\bf 0\}$. In interpreting this result, one has to take into account that by `arbitrary toric ideals' we do not mean arbitrarily \emph{positively graded} toric ideals, but arbitrary in the complete sense. In general, a non-positively graded toric ideal has infinitely many minimal generating sets of different cardinalities and no indispensable binomials, see \cite[Theorems 4.18, 4.19]{CTV1}, and its Graver basis represents the best combinatorial information one could have on such a toric ideal. However, the fact that Graver basis of any toric ideal is in a (precise defined way) one-to-one correspondence to the Graver basis of a hypergraph shouldn't mislead the reader that the latter one is easy to describe. Here, it is important to note that while the Graver bases of toric ideals of graphs have very special form and support structure and are completely understood, see \cite{RTT, TT}, those of hypergraphs can be quite complicated (\cite{PeSt}) and not so many things are known. Under these circumstances, ``controlling" the essential combinatorics of an arbitrary toric ideal by toric ideals of hypergraphs, Theorem 1 is the best one could hope for. 

Theorem 1 has some interesting consequences, two of which we single out as universality results for almost $3$-uniform hypergraphs. These are also presented in Section~\ref{sec:complexity} (Corollaries~\ref{universality_markov} and~\ref{universality}) along with comments  relating them to \cite{DO}.  

\smallskip

The second main result provides  a  polarization-type operation, saying that if $I_A$ is positively graded then there exists a \emph{stable} toric ideal of a hypergraph $\MH$ of same combinatorial complexity as $I_A$. Stable toric ideals, also introduced and extensively studied in \cite{PTV}, are those whose all of the bouquets are non-mixed, or in other words for which passing to the bouquet ideal preserves all combinatorial information and, in the positively graded case, even the homological information. 

\smallskip
\noindent\textbf{Theorem 2} [Theorem \ref{general_complexity}: Stable hypergraphs and positively graded toric ideals]. 
\textit{Let $I_A$ be an arbitrary positively graded nonzero toric ideal. Then there exists a hypergraph $\MH$ such that there is a bijective correspondence between the Graver bases, all minimal Markov bases, all reduced Gr\"obner bases, circuits, and indispensable binomials of $I_A$ and $I_{\MH}$. Moreover the minimal graded free resolution of $I_A$ can be obtained from the corresponding minimal graded free resolution of $I_{\MH}$ and viceversa, and thus all homological data of $I_A$ is inherited from $I_{\MH}$.}
\smallskip

Note by comparison to Theorem 1 that since the toric ideal is positively graded all of the combinatorics, that is all the distinguished special sets, of the toric ideal is encoded by the toric ideal of hypergraph. The great bonus offered by this construction is that one can pass from an arbitrary matrix $A$ with $\Ker_{\ZZ}(A)\cap\NN^n=\{\bf 0\}$ to a $0/1$ matrix $A_{\MH}$ by preserving all the homological data, thus this theorem might be also regarded as a polarization-type operation, see Remark~\ref{homological_data}. Theorem 2 also has some interesting consequences: in particular, it implies that any classification problem about arbitrary positively graded toric ideals involving equality of bases can be reduced to a problem about a toric ideal of a hypergraph defined by a $0/1$-matrix. For example, see the conjecture of Boocher et al. \cite[Question 6.1]{BBDLMNS} discussed at the end of Section~\ref{sec:complexity2} in Remark~\ref{robust_conj}.

\smallskip

The technical backbone of our results are  special types of bouquets for $0/1$ matrices, called bouquets with bases, that encode the basic building blocks of Graver bases elements for hypergraphs. Section~\ref{sec:HypergraphsAndBouquets} offers a more technical motivation for this construction and studies bouquets of $0/1$ matrices in detail. Two main definitions are those of a bouquet with basis and a monomial walk (Definitions~\ref{defi_bouquet} and~\ref{walk}). We classify such bouquets in Theorem~\ref{bouq_basis_classif}. There are two technical applications, which solve two problems for hypergraphs:  Theorem~\ref{connected_sunflower} offers a structural result generalizing \cite{PeSt}, while Proposition~\ref{graver_not_ugb} generalizes a result from \cite{DST} and \cite{TT} from graphs to uniform hypergraphs. In particular, Theorem~\ref{connected_sunflower} and Corollary~\ref{no_free} represent the key ingredients in the construction of the 
hypergraph $\MH$ from Theorem 1. For the convenience of the reader, we begin by revisiting the bouquet constructions and the basic terminology from \cite{PTV} next.


 \section{ Revisiting the bouquet construction } 
\label{sec:BouqDec}

  Let $K$ be a field and $A=[\ab_1,\dots,\ab_n]\in\ZZ^{m\times n}$ be an integer matrix of rank $r$ with the set of column vectors $\{\ab_1,\ldots,\ab_n\}$. We recall that the toric ideal of $A$ is the ideal $I_A\subset K[x_1,\ldots,x_n]$ given by $$I_A=({\bf x}^{{\bf u}^+}-{\bf x}^{{\bf u}^-}: {\bf u}\in\Ker_{\ZZ}(A)),$$
where ${\bf u}={\bf u}^+-{\bf u}^-$ is the unique expression of an integral vector ${\bf u}$ as a difference of two non-negative integral vectors with disjoint support, see \cite[Section 4]{St}. To every integer matrix $A$ one can associate its Gale transform $G(A)$, which is   the $n\times (n-r)$ matrix whose columns span the lattice $\Ker_{\ZZ}(A)$. We will denote the set of ordered row vectors of the matrix $G(A)$ by $\{G({\bf a}_1), \dots, G({\bf a}_n)\}$. The vector ${\bf a}_i$ is called {\em free} if its Gale transform $G({\bf a_i})$ is equal to the zero vector, which means that $i$ is not contained in the support of any minimal  generator of the toric ideal $I_A$, or any element in the Graver basis. A vector which is not free will be called non-free.

\begin{Definition}[{\cite[Definition 1.1]{PTV}}]  
\label{bouq_def}
\rm  The {\em bouquet graph} $G_A$ of $I_A$ is the graph on vertices $\{{\bf a}_1,\dots, {\bf a}_n\}$ whose edge set $E_A$ consists of those $\{\ab_i,\ab_j\}$ for which $G(\ab_i)$ is a rational multiple of $G(\ab_j)$ and vice versa. The connected components of the graph $G_A$ are called {\em bouquets}.
\end{Definition}

It follows from the definition that the free vectors of $A$ form one bouquet, which we call the {\em free bouquet} of $G_A$. A bouquet $B$ which is not free is called non-free. The non-free bouquets are of two types: {\em mixed} and {\em non-mixed}. A non-free bouquet is mixed if contains an edge $\{\ab_i,\ab_j\}$ such that $G(\ab_i)=\lambda G(\ab_j)$ for some $\lambda<0$, and is non-mixed if it is either an isolated vertex or for all of its edges $\{\ab_i,\ab_j\}$ we have $G(\ab_i)=\lambda G(\ab_j)$ with $\lambda>0$, see \cite[Lemma 1.2]{PTV}. By slight abuse of notation, we identify vertices of $G_A$ with their labels; that is, $\ab_i$ will be used to denote vectors in the context of $A$ and $M_A$, and vertices in the context of $G_A$.    

\begin{Example}\label{working_example}
{\em Let $A$ be the matrix 
\[
\left( \begin{array}{ccccccc}
3 & 0 & 0 & 4 & 5 & 0 & 1\\
1 & 1 & 0 & 4 & 5 & 0 & 2\\
3 & 0 & 1 & 0 & 0 & 0 & 0\\
7 & 1 & 2 & 4 & 3 & 1 & 1\\
6 & 0 & 2 & 0 & 0 & 0 & 1
\end{array} \right).
\]
A basis for $\Ker_{\ZZ}(A)$ is given by $(1,2,-3,3,-3,-6,0)$ and $(1,2,-3,-2,1,2,0)$. Thus $G(\ab_1)=(1,1)$, $G(\ab_2)=(2,2)$, $G(\ab_3)=(-3,-3)$, $G(\ab_4)=(3,-2)$, $G(\ab_5)=(-3,1)$, $G(\ab_6)=(-6,2)$ and $G(\ab_7)=(0,0)$. Therefore the bouquet graph $G_A$ with the vertex set $\{\ab_1,\ldots,\ab_7\}$ consists of the  four  bouquets $B_1,\ldots,B_4$ depicted below. 

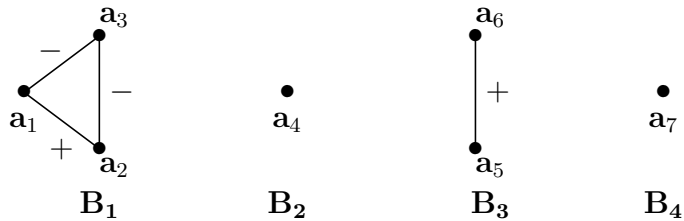
\begin{figure}[hbt]
\begin{center}
\psset{unit=1cm}
\begin{pspicture}(3.75,1.5)(8,4)
\rput(1,3){$\bullet$}
\rput(2,2.25){$\bullet$}
\rput(2,3.75){$\bullet$}
\rput(1,2.6){$\ab_1$}
\rput(2.2,2){$\ab_2$}
\rput(2.2,4){$\ab_3$}
\rput(2,1.5){$\bf{B_1}$}
\rput(1.5,2.25){$+$}
\rput(1.35,3.55){$-$}
\rput(2.3,3){$-$}
\psline[linewidth=0.6pt](2,2.25)(2,3.75)
\psline[linewidth=0.6pt](1,3)(2,2.25)
\psline[linewidth=0.6pt](1,3)(2,3.75)
\rput(4.5,3){$\bullet$}
\rput(7,2.25){$\bullet$}
\rput(7,3.75){$\bullet$}
\rput(4.5,2.6){$\ab_4$}
\rput(7.2,2){$\ab_5$}
\rput(7.2,4){$\ab_6$}
\rput(4.5,1.5){$\bf{B_2}$}
\rput(7.2,1.5){$\bf{B_3}$}
\rput(7.3,3){$+$}
\psline[linewidth=0.6pt](7,2.25)(7,3.75)
\rput(9.5,3){$\bullet$}
\rput(9.5,2.6){$\ab_7$}
\rput(9.5,1.5){$\bf{B_4}$}
\end{pspicture}
\end{center}
\label{Fig1}
\end{figure}

 Here, $B_1$ is mixed, $B_3$ is non-mixed, $B_2$ is non-mixed (since $G(\ab_4)\neq (0,0)$) and $B_4$ is the free bouquet (since $G(\ab_7)=(0,0)$).}
\end{Example}

For each bouquet $B$ there exist some bouquet-index-encoding vectors $\cb_B$ and $\ab_B$ which record the bouquet's  types and linear dependencies. For the complete technical definition of $\cb_B$ we refer the reader to page 9 of \cite{PTV}; here we simply summarize how to compute these vectors and offer an example below. If the bouquet $B$ is free then we set $\cb_B\in\ZZ^n$ to be any nonzero vector such that $\supp(\cb_B)=\{i: \ab_i\in B\}$ and with the property that the first nonzero coordinate is positive. For a non-free bouquet $B$ of $A$, consider the Gale transforms of the elements in $B$. All the elements are nonzero and pairwise linearly dependent, therefore there exists a nonzero coordinate $j$ in all of them. Let $g_j=\gcd(G({\bf a}_i)_j| \ {\bf a}_i\in B)$ and fix the smallest integer $i_0$ such that ${\bf a}_{i_0}\in B$. Then ${\bf c}_B$  is the vector in $\ZZ^n$ whose $i$-th coordinate  is $0$ if ${\bf a}_i\notin B$, and is $\varepsilon_{i_0j}G({\bf a}_i)_j/g_j$ if ${\bf a}_i \in B$, where $\varepsilon_{i_0j}$ represents the sign of the integer $G({\bf a}_{i_0})_j$. With $\cb_B$ introduced the vector $\ab_B$ (see \cite[Definition 1.7]{PTV}), is defined as $\ab_B=\sum_{i=1}^n (c_B)_i\ab_i\in\ZZ^m$, where $(c_B)_i$ is the $i$-th component of the vector $\cb_B$. 

Detecting the type of a non-free bouquet $B$ is important for our further considerations and we can read this information from the corresponding $\cb_B$. Indeed, if $B$ is a  non-free bouquet of $A$, then $B$ is a mixed bouquet if and only if the vector ${\bf c}_B$ has a negative and a positive coordinate, see \cite[Lemma 1.6]{PTV}. Hence the non-free bouquet is non-mixed if and only if the vector $\cb_B$ has all nonzero coordinates positive, taking into account that by definition the first nonzero coordinate of $\cb_B$ is positive. 

The matrix $A_B$ whose column vectors are the vectors $\ab_B$ corresponding to the bouquets of $A$ is called the {\em bouquet matrix} of $A$.  There is a surprising general connection between the matrix $A$ and its bouquet matrix $A_B$ summarized in the following theorem.    

\begin{Theorem}[{\cite[Theorem A]{PTV}}]
\label{all_is_well} 
Let $A=[{\bf a}_1,\dots,{\bf a}_n]\in\ZZ^{ m\times n}$ and its bouquet matrix $A_B=[\ab_{B_1},\dots,\ab_{B_s}]$. There is a bijective correspondence between the elements  
of $\Ker_{\ZZ}(A)$ in general,  and $\Gr(A)$ and $\mathcal C(A)$ in particular,  and the elements of $\Ker_{\ZZ}(A_B)$, and $\Gr(A_B)$ and $\mathcal C(A_B)$, respectively. More precisely, this correspondence is defined as follows: for ${\bf u}=(u_1,\ldots,u_{s})\in\Ker_{\ZZ}(A_B)$ then $B({\bf u})={\bf c}_{B_1}u_1+\cdots+{\bf c}_{B_s}u_s\in\Ker_{\ZZ}(A)$.
\end{Theorem}

\begin{Example}[Example~\ref{working_example}, continued]
{\em
Let $A$ be the matrix from Example~\ref{working_example}. Let us compute the bouquet-index-encoding vectors $\ab_{B_i}$s and $\cb_{B_i}$s. For $\cb_{B_1}$ we can choose $j$ to be either $1$ or $2$, while $i_0=1$. Set $j=1$, then $g_1=\gcd(G(\ab_1)_1,G(\ab_2)_1,G(\ab_3)_1)$ and the nonzero coordinates of $\cb_{B_1}$ are
\[
(c_{B_1})_1=\varepsilon_{11}\frac{G(\ab_1)_1}{g_1}=1, (c_{B_1})_2=\varepsilon_{11}\frac{G(\ab_2)_1}{g_1}=2, (c_{B_1})_3=\varepsilon_{11}\frac{G(\ab_3)_1}{g_1}=-3.
\]
Thus the corresponding bouquet vector is $\cb_{B_1}=(1, 2, -3,0,0,0,0)$ and consequently $\ab_{B_1}=1\ab_1+2\ab_2-3\ab_3=(3,3,0,3,0)$. Similarly one obtains $\cb_{B_2}=(0, 0, 0,1,0,0,0)$, $\ab_{B_2}=\ab_4=(4,4,0,4,0)$, $\cb_{B_3}=(0, 0, 0,0,1,2,0)$, $\ab_{B_3}=1\ab_5+2\ab_6=(5,5,0,5,0)$, $\cb_{B_4}=(0, 0, 0,0,0,0,1)$ and $\ab_{B_4}=\ab_7$. Therefore we obtain the bouquet matrix $A_B=[\ab_{B_1},\ab_{B_2},\ab_{B_3},\ab_{B_4}]\in \ZZ^{5\times 4}$. Now, the correspondence from Theorem~\ref{all_is_well} works as follows. For example, to the vector $(2,1,-2,0)\in\Ker_{\ZZ}(A_B)$ corresponds the vector
\[
B((2,1,-2,0))=2\cb_{B_1}+1\cb_{B_2}-2\cb_{B_3}+0\cb_{B_4}=(2,4,-6,1,-2,-4,0)\in\Ker_{\ZZ}(A).
\]

Note that the bouquet ideal $I_{A_B}$ is also defined by the matrix 
\[
\left( \begin{array}{cccc}
3 & 4 & 5 & 0\\
0 & 0 & 0 & 1
\end{array} \right),
\]
and taking into account that $\ab_{B_4}$ is a free vector, $I_A$ is just the extension of the defining ideal of the monomial curve parametrized by $(t^3,t^4,t^5)$ in the polynomial ring in four indeterminates. 
}
\end{Example}

A {\em subbouquet} of $G_A$ is an induced subgraph of $G_A$. One easily sees that subbouqets are cliques and maximal ones are  the bouquets defined above. $A$ is said to have a \emph{subbouquet decomposition} if there exists a family of subbouquets $B_1,\ldots,B_t$ that is pairwise vertex-disjoint, the union of whose vertices  equals $\{\ab_1,\ldots,\ab_n\}$. A subbouquet decomposition always exists if we consider, for example, the subbouquet decomposition induced by all of the bouquets. We define vectors $\cb_B$ and $\ab_B$ for each subbouquet $B$ as we did for the bouquets. The matrix associated to such a subbouquet decomposition is called a subbouquet matrix. Crucially, Theorem \ref{all_is_well} is true even if we replace bouquets with proper subbouquets which form a subbouquet decomposition of $A$. 

\begin{Example}[Example~\ref{working_example}, continued]
\label{subbouquet_dec}
{\em The family of subbouquets $B'_1,\ldots,B'_5$ depicted below represents a subbouquet decomposition of $A$. Note that all of the non-free bouquets $B'_1,\ldots,B'_4$ are non-mixed. 
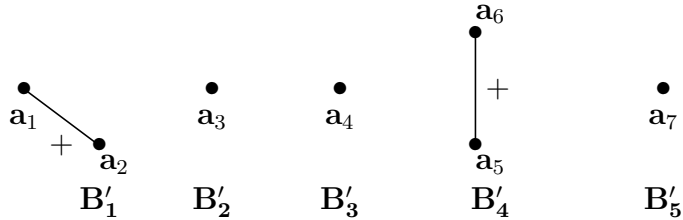
\begin{figure}[hbt]
\begin{center}
\psset{unit=1cm}
\begin{pspicture}(3.75,1.5)(8,4)
\rput(1,3){$\bullet$}
\rput(2,2.25){$\bullet$}
\rput(3.5,3){$\bullet$}
\rput(1,2.6){$\ab_1$}
\rput(2.2,2){$\ab_2$}
\rput(3.5,2.6){$\ab_3$}
\rput(2,1.5){$\bf{B'_1}$}
\rput(3.5,1.5){$\bf{B'_2}$}
\rput(1.5,2.25){$+$}
\psline[linewidth=0.6pt](1,3)(2,2.25)
\rput(5.2,3){$\bullet$}
\rput(7,2.25){$\bullet$}
\rput(7,3.75){$\bullet$}
\rput(5.2,2.6){$\ab_4$}
\rput(7.2,2){$\ab_5$}
\rput(7.2,4){$\ab_6$}
\rput(5.2,1.5){$\bf{B'_3}$}
\rput(7.2,1.5){$\bf{B'_4}$}
\rput(7.3,3){$+$}
\psline[linewidth=0.6pt](7,2.25)(7,3.75)
\rput(9.5,3){$\bullet$}
\rput(9.5,2.6){$\ab_7$}
\rput(9.5,1.5){$\bf{B'_5}$}
\end{pspicture}
\end{center}
\label{Fig2}
\end{figure}
}
\end{Example}
There is a natural question of whether there is an inverse construction to the one described in Theorem~\ref{all_is_well}: given a set of vectors $\ab_1,\ldots,\ab_s$ and $\cb_1,\ldots,\cb_s$ that can act as bouquet-index-encoding vectors can we construct a toric ideal $I_A$ whose $s$ (sub)bouquets are encoded by the given vectors? The answer is yes, comprised in the next theorem, and will be often exploited in the next sections.  

\begin{Theorem}[{\cite[Theorem B]{PTV}}]
\label{inverse_construction}
Let $\{\ab_1,\ldots,\ab_s\}\subset\ZZ^m$ be an arbitrary set of vectors. Let $\cb_1,\ldots,\cb_s$ be any set of primitive vectors, with $\cb_i\in\ZZ^{m_i}$ for some $m_i\geq 1$, each having full support and a positive first coordinate.  Then, there exists a matrix $A$ with the subbouquet decomposition $B_1,\ldots,B_s$, such that the $i^{th}$ subbouquet	is encoded by the following vectors: $\ab_{B_i}=(\ab_i,{\bf 0},\ldots,{\bf 0})$ and $\cb_{B_i}=({\bf 0},\ldots,\cb_i,\ldots,{\bf 0})$, where the support of $\cb_{B_i}$ is precisely in the $i^{th}$ block, of size $m_i$.
\end{Theorem}

As a consequence of Definition~\ref{bouq_def} and the fact there are  two types of non-free bouquets, one has the following straightforward combinatorial classification of toric ideals. A toric ideal may have: 1) all of its non-free bouquets non-mixed; 2) all of its non-free bouquets mixed, or 3) the non-free bouquets are either mixed or non-mixed. In the first case, the Theorem below says that all of the combinatorial information is preserved when passing from $I_A$ to $I_{A_B}$. In addition,   homological properties are also preserved; see \cite[Theorem 3.11]{PTV}.    

\begin{Theorem}[{\cite[Theorem C]{PTV}}]
\label{stable_toric} 
Let $I_A$ be a stable toric ideal, that is all of the non-free bouquets of $A$ are non-mixed. Then the bijective correspondence between the elements of $\Ker_{\ZZ}(A)$ and $\Ker_{\ZZ}(A_B)$ given by ${\bf u}\mapsto B({\bf u})$, is preserved when we restrict to any of the following sets: Graver basis, circuits, indispensable binomials, minimal Markov bases, reduced Gr\"obner bases (universal Gr\"obner basis).
\end{Theorem}

This result holds in a more general setting, as explained in detail in \cite[Section 3]{PTV}. For example, even if $I_A$ is not stable, one cand find a (maximal) subbouquet decomposition such that all of its subbouquets are non-mixed, and whose subbouquet matrix $A_B$ has the property that the bijective correspondence between $\Ker_{\ZZ}(A)$ and $\Ker_{\ZZ}(A_B)$ from Theorem~\ref{stable_toric} is preserved. In particular, Example~\ref{subbouquet_dec} gives such a maximal subbouquet decomposition of the matrix $A$ from Example~\ref{working_example}.

The second class of toric ideals, described above, and whose bouquets are all mixed have the following nice property. 
\begin{Proposition}[{\cite[Corollary 4.4]{PTV}}]
\label{gen_no_free}
Suppose that every non-free bouquet of $A$ is mixed. Then $I_A$ is strongly robust,  i.e. the following sets coincide:
\begin{itemize}
\item the Graver basis of $A$,
\item the universal Gr{\"o}bner basis of $A$,
\item any reduced Gr{\"o}bner basis of $A$,
\item any minimal Markov basis of $A$.
\end{itemize}
\end{Proposition}

However, despite the fact that almost all the examples of strongly robust ideals, previously known in the literature, had all of the bouquets mixed, there are examples of strongly robust ideals which have both types of non-free bouquets, see \cite[Example 4.3. b)]{PTV}. This lead us to raise the following question: Is it true that a strongly robust ideal has at least one mixed bouquet?, see \cite[Question 4.6]{PTV}. Sullivant answered this question in the affirmative in \cite{Su} for the particular case of codimension 2 strongly robust toric ideals.

\section{Hypergraphs and bouquets with bases}
\label{sec:HypergraphsAndBouquets}

As this section consists of three parts, we begin with a `roadmap' to point out the main definitions and results. As described in the Introduction, this section is concerned with toric ideals of  $0/1$ matrices, which are, by definition, incidence matrices of hypergraphs. To understand this class of toric ideals, we need two main definitions: a \emph{bouquet with basis} in Definition~\ref{defi_bouquet},  lacking in the previous literature and a \emph{monomial walk} in Definition~\ref{walk},  recovering the meaning of monomial walks on graphs from \cite{RTT}, \cite{Vill}. The first two subsections are motivated by \cite[Problem 6.2]{PeSt}, and they, essentially, partially solve its generalization. Theorem~\ref{bouq_basis_classif} classifies bouquets with bases,  which are either free or mixed subbouquets,  while the running example shows its implications on identifying and  constructing elements in the Graver basis of a hypergraph.  Proposition~\ref{connected_sunflower} shows that certain hypergraphs based on a sunflower are bouquets with bases and, in particular, recovers the main structural result of \cite{PeSt}. As another application of bouquets with bases, Proposition~\ref{graver_not_ugb} generalizes a result by \cite{DST} and \cite{TT} for graphs to uniform hypergraphs and showing that the universal Gr{\"o}bner and Graver bases differ for complete uniform hypergraphs with enough vertices. However, in general, there exist hypergraphs that do not admit such a subbouquet structure.  In this case, we provide general constructions in Sections~\ref{sec:complexity} and \ref{sec:complexity2}.

Let $\MH=(V,\ME)$ be a finite hypergraph on the set of vertices $V=\{x_1,\dots,x_m\}$ with edge set $\ME=\{E_1,\dots,E_n\}$, where each $E_i$ is a subset of $\{x_1,\dots,x_m\}$. When $|E_i|=2$ for all $i$ and $E_i\neq E_j$ for all $i\neq j$, we have a finite simple graph, and we will specialize to this case to  recall  results about graphs. We denote by ${\pmb \alpha}_E$ the support (column) vector of an edge $E$, and thus the toric ideal $I_{\MH}$ is the toric ideal of the matrix  $[{\pmb\alpha}_{E_1},\ldots,{\pmb\alpha}_{E_n}]$. 
Hereafter, for ease of notation, various bases of $I_\MH$ will be referred to as bases of $\MH$; the reader may simply keep in mind that the underlying toric matrix is the incidence matrix of the hypergraph $\MH$; for example, $Gr(\MH):=Gr(A_\MH)$ where $A_\MH$ is the vertex-edge incidence matrix of $\MH$.

\subsection{Bouquets with bases}

Let  $\MH=(V,\ME)$ be a hypergraph and $U\subset V$. We define the multiset $\MU_{\ME}=\{E\cap U|E\in \ME, \ E\cap U\neq\emptyset\}$ and the set $\ME_U=\{E\in \ME| E\cap U\not= \emptyset \}$. Note that the multiset $\MU_{\ME}$ and the set $\ME_U$ have the same number of elements: $\ME_U$ is the set of edges of $\MH$ that intersect the vertex set $U$, while $\MU_\ME$ consists of their  restrictions to $U$.

\begin{Definition}\label{defi_bouquet}
\rm
The set of edges $\ME_U$ of $\MH$ is called a {\em bouquet with basis} $U\subset V$ if the toric ideal of the multi-hypergraph  $(U,\MU_{\ME})$ is principal, generated by an element $e^{{\bf c}^+}-e^{{\bf c}^-}$ with $\supp({\bf c})=\ME_{U}$, and such that the first nonzero coordinate of ${\cb}={\cb}^+-{\cb}^-$ is positive. Here the toric ideal of $(U,\MU_{\ME})$ is contained in the polynomial ring with variables indexed by the non-empty subsets $e_i:=E_i\cap U$. 

 Moreover, for a bouquet with basis $\ME_U$, we define the vector ${\bf c}_{\ME_U}\in\ZZ^n$ such that $({\bf c}_{\ME_U})_E=c_E$ for any ${E\in\ME_U}$ and $0$ otherwise, and the vector ${\bf a}_{\ME_U}=\sum_{E\in\ME_U} c_E{\pmb\alpha}_E\in\ZZ^m$. Here $c_E$ denotes the coordinate of $\cb$ corresponding to the edge $E$, and note from the definition of $\cb$ that $c_E\neq 0$ for all $E\in\ME_{U}$.  
\end{Definition}

The following examples capture the subtleties of this definition and its similarity to Definition~\ref{bouq_def}.  

\begin{Example}\label{bouquet_basis_examples}
{\em a) Consider the hypergraph $\MH=(V,\ME)$ from Figure~\ref{three_bouq}, with the set of vertices $V=\{x,v_1,\ldots,v_{22}\}$ and whose set of edges $\ME$ consists of the following $20$ edges: $E_1=\{x, v_1, v_2\}$, $E_2=\{x, v_3, v_4\}$, $E_3=\{x, v_5, v_6\}$, $E_4=\{ v_1, v_3, v_5\}$, $E_5=\{v_2, v_4, v_6\}$, $E_6=\{x, v_7, v_8\}$, $E_7=\{x, v_9, v_{10}\}$, $E_8=\{x, v_{11}, v_{12}\}$, $E_9=\{ x, v_{13}, v_{14}\}$, $E_{10}=\{v_7, v_8, v_9\}$,  $E_{11}=\{v_{10}, v_{11}, v_{13}\}$, $E_{12}=\{v_{12}, v_{14}\}$, $E_{13}=\{x, v_{15}, v_{16}\}$, $E_{14}=\{x, v_{17}, v_{18}\}$, $E_{15}=\{x, v_{19}, v_{20}\}$, $E_{16}=\{ x, v_{21}, v_{22}\}$, $E_{17}=\{v_{16}, v_{18}, v_{20}\}$,  $E_{18}=\{ v_{15}, v_{19}\}$, $E_{19}=\{v_{17}, v_{19}, v_{21}\}$, $E_{20}=\{ v_{20}, v_{22}\}$.   

The first bouquet with basis $\ME_{U_1}$ has five edges, the basis $U_1$ is the set $\{v_1,\ldots,v_6\}$, the vector ${\bf c}_1$ corresponds to the binomial generator $e_1e_2e_3-e_4e_5$ of the toric ideal of $(U_1,\MU_{1_{\ME}})$, and thus ${\bf c}_{\ME_{U_1}}=(1,1,1,-1,-1,0,\ldots,0)\in\ZZ^{20}$ and  $${\bf a}_{\ME_{U_1}}={\pmb\alpha}_{E_1}+{\pmb\alpha}_{E_2}+{\pmb\alpha}_{E_3}-{\pmb\alpha}_{E_4}-{\pmb\alpha}_{E_5}=(3,0,\ldots,0)\in\ZZ^{ 23}.$$

The second bouquet with basis $\ME_{U_2}$ has seven edges, the basis $U_2$ is the set $\{v_7,\ldots,v_{14}\}$, the vector ${\bf c}_2$ corresponds to the binomial generator $e_6e_7e_8e_9-e_{10}e_{11}e_{12}$ of the toric ideal of $(U_2,\MU_{2_{\ME}})$, and thus the encoding vectors are $${\bf c}_{\ME_{U_2}}=(0,0,0,0,0,1,1,1,1,-1,-1,-1,0,\ldots,0)\in\ZZ^{20}$$ and $${\bf a}_{\ME_{U_2}}=\sum_{i=6}^9 {\pmb\alpha}_{E_i}-\sum_{j=10}^{12} {\pmb\alpha}_{E_j}=(4,0,\ldots,0)\in\ZZ^{ 23}.$$

The third bouquet with basis $\ME_{U_3}$ has eight edges, the basis $U_3=\{ v_{15},\ldots,v_{22}\}$, the vector ${\bf c}_3$ corresponds to the binomial generator $e_{13}e_{14}e_{15}^2e_{16}-e_{17}e_{18}e_{19}e_{20}$ of the toric ideal of $(U_3,\MU_{3_{\ME}})$, and thus ${\bf c}_{\ME_{U_3}}=(0,\ldots,0,1,1,2,1,-1,-1,-1,-1)\in\ZZ^{20}$ and $${\bf a}_{\ME_{U_3}}={\pmb\alpha}_{E_{13}}+{\pmb\alpha}_{E_{14}}+2{\pmb\alpha}_{E_{15}}+{\pmb\alpha}_{E_{16}}-\sum_{j=17}^{20}{\pmb\alpha}_{E_j} =(5,0,\ldots,0)\in\ZZ^{ 23}.$$ 

On the other hand, if $A=A_{\MH}$ is the incidence matrix of the hypergraph $\MH$ with columns ${\pmb\alpha}_{E_1},\ldots,{\pmb\alpha}_{E_{20}}$, then $G_A$ has three mixed non-free bouquets $B_1,B_2,B_3$. More precisely, the first bouquet $B_1$ corresponds to the vectors ${\pmb\alpha}_{E_1},\ldots,{\pmb\alpha}_{E_5}$, the second bouquet $B_2$ to the vectors ${\pmb\alpha}_{E_6},\ldots,{\pmb\alpha}_{E_{12}}$, and the third bouquet $B_3$ to the vectors ${\pmb\alpha}_{E_{13}},\ldots,{\pmb\alpha}_{E_{20}}$. Moreover, it turns out that $\cb_{B_i}=\cb_{\ME_{U_i}}$ and $\ab_{B_i}=\ab_{\ME_{U_i}}$  for all $i$. 

\begin{figure}[hbt]
\label{three_bouq}
\begin{center}
\psset{unit=0.9cm}
\begin{pspicture}(0,0.75)(12,3.5)
 \rput(1,1){$\bullet$}
 \rput(1,2){$\bullet$}
 \rput(1,3){$\bullet$}
 \rput(2,1.5){$\bullet$}
 \rput(2,2){$\bullet$}
 \rput(2,2.5){$\bullet$}
 \rput(3,2){$\bullet$}
 \rput(2,3){\tiny $v_1$}
 \rput(1,3.5){\tiny $v_2$}
 \rput(1.6,2){\tiny $v_3$}
 \rput(0.35,2){\tiny $v_4$}
 \rput(2,1){\tiny $v_5$}
 \rput(1,0.5){\tiny $v_6$}
 \rput(3.5,2){\tiny $x$}
\psecurve[linewidth=1.1pt, linecolor=black](1,0.8)(1.35,2)(1,3.2)(0.65,2)(1,0.8)(1.35,2)(1,3.2)(0.65,2)
\psecurve[linewidth=1.1pt, linecolor=black](2,1.3)(2.2,2)(2,2.7)(1.8,2)(2,1.3)(2.2,2)(2,2.7)(1.8,2)
\psecurve[linewidth=1.1pt, linecolor=black](3.2,2)(2,2.2)(0.8,2)(2,1.8)(3.2,2)(2,2.2)(0.8,2)(2,1.8)
\psecurve[linewidth=1.1pt, linecolor=black](3.2,1.9)(2.15,2.65)(0.8,3.1)(1.85,2.35)(3.2,1.9)(2.15,2.65)(0.8,3.1)(1.85,2.35)
\psecurve[linewidth=1.1pt, linecolor=black](3.2,2.1)(1.85,1.65)(0.8,0.9)(2.15,1.35)(3.2,2.1)(1.85,1.65)(0.8,0.9)(2.15,1.35)
 
 \rput(4.5,2){\tiny $x$}
 \rput(5,2){$\bullet$}
 \rput(5,2.5){$\bullet$}
 \rput(5,3){$\bullet$}
 \rput(6,1.5){$\bullet$}
 \rput(6,2){$\bullet$}
 \rput(6,2.5){$\bullet$}
 \rput(7,1){$\bullet$}
 \rput(7,2){$\bullet$}
 \rput(7,3){$\bullet$}
 \rput(4.5,2.5){\tiny $v_7$}
 \rput(4.5,3){\tiny $v_8$}
 \rput(7.5,3){\tiny $v_9$}
 \rput(6.4,2.65){\tiny $v_{10}$}
 \rput(6.5,2){\tiny $v_{11}$}
 \rput(7.5,2){\tiny $v_{12}$}
 \rput(6,1){\tiny $v_{13}$}
 \rput(7.5,1){\tiny $v_{14}$} 
\psecurve[linewidth=1.1pt, linecolor=black](7,0.8)(7.2,1.5)(7,2.2)(6.8,1.5)(7,0.8)(7.2,1.5)(7,2.2)(6.8,1.5)
\psecurve[linewidth=1.1pt, linecolor=black](6,1.3)(6.2,2)(6,2.7)(5.8,2)(6,1.3)(6.2,2)(6,2.7)(5.8,2)
\psecurve[linewidth=1.1pt, linecolor=black](4.8,2)(6,1.8)(7.2,2)(6,2.2)(4.8,2)(6,1.8)(7.2,2)(6,2.2)
\psecurve[linewidth=1.1pt, linecolor=black](4.8,1.9)(5.85,2.65)(7.2,3.1)(6.15,2.35)(4.8,1.9)(5.85,2.65)(7.2,3.1)(6.15,2.35)
\psecurve[linewidth=1.1pt, linecolor=black](4.8,2.1)(6.15,1.65)(7.2,0.9)(5.85,1.35)(4.8,2.1)(6.15,1.65)(7.2,0.9)(5.85,1.35)
\psecurve[linewidth=1.1pt, linecolor=black](5,1.8)(5.2,2.5)(5,3.2)(4.8,2.5)(5,1.8)(5.2,2.5)(5,3.2)(4.8,2.5)
\psecurve[linewidth=1.1pt, linecolor=black](5.15,2.75)(5,3.2)(4.85,2.75)(5,2.3)(5.15,2.8)(5,3.2)
\psline[linewidth=1.1pt, linecolor=black](5.15,2.8)(7,2.8)
\psline[linewidth=1.1pt, linecolor=black](5,3.2)(7,3.2)
\pscurve[linewidth=1.1pt, linecolor=black](7,2.8)(7.2,3)(7,3.2)

 \rput(8.5,2){\tiny $x$}
 \rput(9,2){$\bullet$}
 \rput(9,2.5){$\bullet$}
 \rput(9,3){$\bullet$}
 \rput(10,1.5){$\bullet$}
 \rput(10,2){$\bullet$}
 \rput(10,2.5){$\bullet$}
 \rput(11,1){$\bullet$}
 \rput(11,2){$\bullet$}
 \rput(11,3){$\bullet$}
 \rput(8.5,2.5){\tiny $v_{15}$}
 \rput(8.5,3){\tiny $v_{16}$}
 \rput(10.4,2.65){\tiny $v_{17}$}
 \rput(11.5,3){\tiny $v_{18}$}
 \rput(10.5,2){\tiny $v_{19}$}  
 \rput(11.65,2){\tiny $v_{20}$}
 \rput(10,1){\tiny $v_{21}$}
 \rput(11.5,1){\tiny $v_{22}$} 
 \psecurve[linewidth=1.1pt, linecolor=black](11,0.8)(11.2,1.5)(11,2.2)(10.8,1.5)(11,0.8)(11.2,1.5)(11,2.2)(10.8,1.5)
\psecurve[linewidth=1.1pt, linecolor=black](10,1.3)(10.2,2)(10,2.7)(9.8,2)(10,1.3)(10.2,2)(10,2.7)(9.8,2)
\psecurve[linewidth=1.1pt, linecolor=black](8.7,2)(10,1.7)(11.3,2)(10,2.3)(8.7,2)(10,1.7)(11.3,2)(10,2.3)
\psecurve[linewidth=1.1pt, linecolor=black](8.8,1.9)(9.85,2.65)(11.2,3.1)(10.15,2.35)(8.8,1.9)(9.85,2.65)(11.2,3.1)(10.15,2.35)
\psecurve[linewidth=1.1pt, linecolor=black](8.8,2.1)(10.15,1.65)(11.2,0.9)(9.85,1.35)(8.8,2.1)(10.15,1.65)(11.2,0.9)(9.85,1.35)
\psecurve[linewidth=1.1pt, linecolor=black](9,1.8)(9.2,2.5)(9,3.2)(8.8,2.5)(9,1.8)(9.2,2.5)(9,3.2)(8.8,2.5)
\psecurve[linewidth=1.1pt, linecolor=black](8.8,2.6)(9.65,2.4)(10.2,1.9)(9.35,2.1)(8.8,2.6)(9.65,2.4)(10.2,1.9)(9.35,2.1)
\psecurve[linewidth=1.1pt, linecolor=black](10.85,2.5)(11,3.2)(11.15,2.5)(11,1.8)(10.85,2.8)(11,3.2)
\psline[linewidth=1.1pt, linecolor=black](10.85,2.8)(9,2.8)
\psline[linewidth=1.1pt, linecolor=black](11,3.2)(9,3.2)
\pscurve[linewidth=1.1pt, linecolor=black](9,2.8)(8.8,3)(9,3.2)

\end{pspicture}
\end{center}
\caption{Bouquets with bases $\ME_{U_1},\ME_{U_2},\ME_{U_3}$.}\label{three_bouq} 
\end{figure}
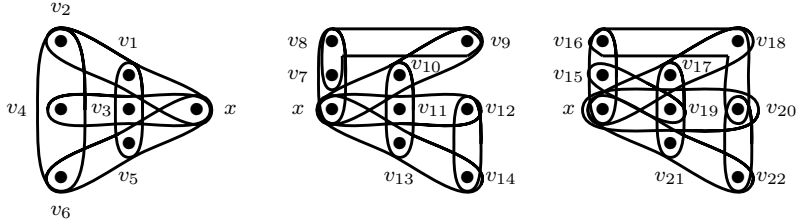

(b) The tetrahedron is an example of a bouquet with basis, and the basis can be chosen to be any facet. Let $V=\{ v_1, v_2, v_3, v_4\}$ and $\ME= \{E_1=\{v_2, v_3, v_4\},\  E_2=\{v_1, v_3, v_4\},\ E_3=\{v_1, v_2, v_4\},\ E_4=\{v_1, v_2, v_3\} \}$. If we consider as basis the set $U_1=\{v_1, v_2, v_3\}$ then
$\MU_{1_{\ME}}=\{e_1=\{v_2, v_3\}, e_2=\{v_1, v_3\}, e_3=\{v_1, v_2\}, e_4=\{v_1, v_2, v_3\} \}$ and the toric ideal of $(U_1, \MU_{1_{\ME}})$  is generated by the element $e^{{\bf c}^+}-e^{{\bf c}^-}=e_1e_2e_3-e_4^2$, with ${\bf c}=(1,1,1,-2)$. Therefore, ${\bf c}_{\ME_{U_1}}=(1,1,1,-2)$ and ${\bf a}_{\ME_{U_1}}={\pmb\alpha}_{E_1}+{\pmb\alpha}_{E_2}+{\pmb\alpha}_{E_3}-2{\pmb\alpha}_{E_4}=(0,0,0,3)$. The same example has four different representations as a bouquet with basis $U$, each one of them having pairwise distinct vectors $\cb_{\ME_U}$ and $\ab_{\ME_U}$, respectively. More precisely, if $U_2=\{v_1, v_2, v_4\}$ then ${\bf c}_{\ME_{U_2}}=(1,1,-2,1)$ and ${\bf a}_{\ME_{U_2}}=(0,0,3,0)$, if $U_3=\{v_1, v_3, v_4\}$ then ${\bf c}_{\ME_{U_3}}=(1,-2,1,1)$ and ${\bf a}_{\ME_{U_3}}=(0,3,0,0)$, and if $U_4=\{v_2, v_3, v_4\}$ then ${\bf c}_{\ME_{U_4}}=(2,-1,-1,-1)$ and ${\bf a}_{\ME_{U_4}}=(-3,0,0,0)$.  Note that $I_\MH$ is the zero ideal and the Gale transforms $G({\pmb\alpha}_{E_i})$ are zero vectors for all $i$. Therefore, the bouquet graph of $I_{\MH}$ has only the free bouquet $B$ consisting of all vertices ${\pmb\alpha}_{E_1},\ldots,{\pmb\alpha}_{E_4}$. In particular, by definition of $\cb_B$ and $\ab_B$ we observe that we can choose these vectors to be any of the 4 pairs of vectors obtained before as $\cb_{\ME_U}$ and $\ab_{\ME_U}$.}
\end{Example}


Given the apparent similarity with Definition~\ref{bouq_def} and in light of the two situations discussed in Example~\ref{bouquet_basis_examples}, one would expect a certain relationship between a bouquet with basis of a hypergraph $\MH$ and a bouquet of its incidence matrix. Even more, one may ask whether $\cb_{\ME_U}$ and $\ab_{\ME_U}$, the encoding vectors of a bouquet with basis, are analogous to  $\cb_B$ and $\ab_B$ as defined in Section~\ref{sec:BouqDec}. The following result clarifies the `bouquet with basis' terminology. As we will see, a bouquet with basis corresponds to a subbouquet of the incidence matrix of the hypergraph, where the correspondence is the natural one associating to a set of edges of a hypergraph the set of corresponding support column vectors of its incidence matrix. Thus, from now on, by abuse of notation, we will identify the bouquet with basis $\ME_U$ with the corresponding subbouquet of the incidence matrix. Furthermore, we also note that $\cb_{\ME_U}$ matches the definition of $\cb_B$ from Section~\ref{sec:BouqDec}, and the same holds for $\ab_{\ME_U}$.

\begin{Theorem}\label{bouq_basis_classif} 
A bouquet with basis of the hypergraph $\MH=(V,\ME)$ is either a free subbouquet or a mixed subbouquet of the incidence matrix of $\MH$. 
\end{Theorem}  
\begin{proof}
Let $\ME_U$ be a bouquet with basis for some $U\subset V$, and set $\ME_U=\{E_1,\ldots,E_s\}$. Furthermore, denote by $M\in\ZZ^{m\times n}$ the incidence matrix of the hypergraph $(V,\ME)$. We may arrange $M$ so that its rows are indexed first by vertices in $U$ and then  vertices in $V\setminus U$, while its columns are indexed first by the edges $E_1,\ldots, E_s$ and next by the remaining edges, if any. Note that the submatrix of $M$ corresponding to the rows indexed by $U$ and the first $s$ columns, denoted by $M_U$, is the incidence matrix of the multi-hypergraph $(U,\MU_E)$, while the submatrix of $M$ corresponding to the rows indexed by $U$ and the rest of the columns is $0$. Finally, denote by $G=(g_{ij})\in\ZZ^{n\times r}$ the Gale transform of $M$, and according to the labeling of the columns of $M$ its first $s$ rows are $G({\pmb\alpha}_{E_1}),\ldots, G({\pmb\alpha}_{E_s})$. By definition of the Gale transform, column $j$ of $G$ is in the kernel of $M$. Therefore $(g_{1j},\ldots,g_{sj})\in\Ker M_U$ and since $\ME_U$ is a bouquet with basis then $(g_{1j},\ldots,g_{sj})$ is a multiple of ${\bf c}=(c_{E_1},\ldots,c_{E_s})$. Thus $g_{ij}=\lambda_j c_{E_i}$ for all $i,j$ and implicitly we obtain $G({\pmb\alpha}_{E_i})=c_{E_i} (\lambda_1,\lambda_2,\ldots,\lambda_r)$ for all $i\in\{1,\ldots,s\}$. Then it follows at once that ${\pmb\alpha}_{E_1},\ldots,{\pmb\alpha}_{E_s}$ belong to the same subbouquet $B$ of $M$. In addition, since $\supp(\cb)=\ME_U$ then $c_{E_i}\neq 0$ for all $i\in\{1,\ldots,s\}$. Thus there are two possibilities: $G({\pmb\alpha}_{E_i})={\bf 0}$ for all $i$ or $G({\pmb\alpha}_{E_i})\neq{\bf 0}$ for all $i$. In the first case we obtain that $B$ is a free subbouquet, while in the second $B$ is a non-free subbouquet. Moreover, the toric ideal of $(U,\MU_E)$ being positively graded implies that the vector ${\bf c}_{\ME_U}$ has at least one positive and one negative coordinate, and by the description of a non-free mixed bouquet $B$ in terms of $\cb_B$ from Section~\ref{sec:BouqDec},  we obtain that $B$ is mixed. \qed

\end{proof}

Let us point out that a bouquet with basis can be a proper subbouquet of a bouquet. Indeed, let $A'$ be the submatrix of the incidence matrix of the hypergraph $\MH$ from Example~\ref{bouquet_basis_examples}(a) corresponding to the first twelve columns. As it was already noticed the first five edges form a bouquet with basis, and the other seven edges form another bouquet with basis. In contrast, one can easily see that the bouquet graph of $A'$ has one mixed bouquet - consisting of all  column vectors of $A'$, and thus the two bouquets with bases are proper subbouquets.  

The previous theorem shows that bouquets with bases are always subbouquets. The natural converse question arises: Can there exist (sub)bouquets of the incidence matrix of a hypergraph which are not bouquets with bases? The answer is yes, and as an example consider the complete graph $K_4$ on four vertices, whose incidence matrix is the submatrix corresponding to the first 6 columns of $A$ from \cite[Example 1.4]{PTV}. As it was shown there $K_4$ has three non-mixed bouquets, each one being an edge. On the other hand, $K_4$ does not have any bouquets with  bases, since by Theorem~\ref{bouq_basis_classif} this would imply the existence either of a mixed non-free subbouquet or a free subbouquet.   

The main consequence of Theorem~\ref{bouq_basis_classif} is that if the edge set of a hypergraph $\MH$ can be partitioned into bouquets with bases, then the toric ideal $I_{\MH}$ is easier to describe, via Theorem~\ref{all_is_well}. The following example captures this remark.   

\begin{Example}[Example~\ref{bouquet_basis_examples}(a), continued]
\label{graver_bouquets_bases}
{\em The hypergraph $\MH$ has three bouquets with bases $\ME_{U_1},\ME_{U_2},\ME_{U_3}$ which partition $\ME$. Therefore, the subbouquet ideal of $A_{\MH}$ is given by the toric ideal of the matrix whose columns are ${\bf a}_{\ME_{U_1}}=(3,0,\ldots,0),{\bf a}_{\ME_{U_2}}=(4,0,\ldots,0),{\bf  a}_{\ME_{U_3}}=(5,0,\ldots,0)\in\ZZ^{23}$, which is the same as the toric ideal of the monomial curve $A=(3 \ 4\ 5 )$. Computing with \cite{4ti2} we obtain that the Graver basis of $I_A$ consists of seven elements  $(4, -3,  0)$, $ (1, -2,  1)$, $ (3, -1, -1)$, $ (2,  1, -2)$, $ (5,  0, -3) $,  $ (1,  3, -3)$, $ (0,  5, -4)$. Therefore, by Theorem~\ref{all_is_well}, the Graver basis of the toric ideal of the hypergraph consists of seven elements. For example,  $(1,-2,1)$ corresponds to the following Graver basis element of $I_{\MH}$: $${\bf c}_{\ME_{U_1}}-2{\bf c}_{\ME_{U_2}}+{\bf c}_{\ME_{U_3}}=(1,1,1,-1,-1,-2,-2,-2,-2,2,2,2,1,1,2,1,-1,-1,-1,-1)$$ and it encodes the binomial $$E_1E_2E_3E_{10}^2E_{11}^2E_{12}^2E_{13}E_{14}E_{15}^2E_{16} -
 E_4E_5E_6^2E_7^2E_8^2E_9^2E_{17}E_{18}E_{19}E_{20}.$$
This binomial corresponds to the primitive monomial walk (see Definition~\ref{walk}) depicted in Figure~\ref{Fig8}, where the three copies of the vertex $x$ should be identified, but are shown separately for better visibility.

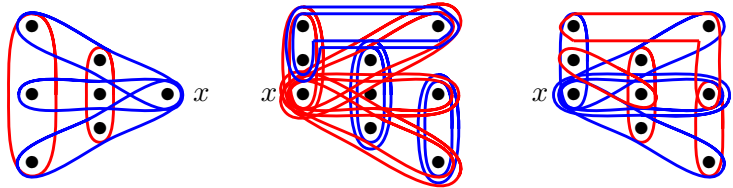
\begin{figure}[hbt]
\label{Fig8}
\begin{center}
\psset{unit=0.9cm}
\begin{pspicture}(0,0.75)(12,3.5)
 \rput(1,1){$\bullet$}
 \rput(1,2){$\bullet$}
 \rput(1,3){$\bullet$}
 \rput(2,1.5){$\bullet$}
 \rput(2,2){$\bullet$}
 \rput(2,2.5){$\bullet$}
 \rput(3,2){$\bullet$}
 \rput(3.5,2){\small $x$}
\psecurve[linewidth=1.1pt, linecolor=red](1,0.8)(1.35,2)(1,3.2)(0.65,2)(1,0.8)(1.35,2)(1,3.2)(0.65,2)
\psecurve[linewidth=1.1pt, linecolor=red](2,1.3)(2.2,2)(2,2.7)(1.8,2)(2,1.3)(2.2,2)(2,2.7)(1.8,2)
\psecurve[linewidth=1.1pt, linecolor=blue](3.2,2)(2,2.2)(0.8,2)(2,1.8)(3.2,2)(2,2.2)(0.8,2)(2,1.8)
\psecurve[linewidth=1.1pt, linecolor=blue](3.2,1.9)(2.15,2.65)(0.8,3.1)(1.85,2.35)(3.2,1.9)(2.15,2.65)(0.8,3.1)(1.85,2.35)
\psecurve[linewidth=1.1pt, linecolor=blue](3.2,2.1)(1.85,1.65)(0.8,0.9)(2.15,1.35)(3.2,2.1)(1.85,1.65)(0.8,0.9)(2.15,1.35)
 
 \rput(4.5,2){\small $x$}
 \rput(5,2){$\bullet$}
 \rput(5,2.5){$\bullet$}
 \rput(5,3){$\bullet$}
 \rput(6,1.5){$\bullet$}
 \rput(6,2){$\bullet$}
 \rput(6,2.5){$\bullet$}
 \rput(7,1){$\bullet$}
 \rput(7,2){$\bullet$}
 \rput(7,3){$\bullet$}
\psecurve[linewidth=1.1pt, linecolor=blue](7,0.8)(7.2,1.5)(7,2.2)(6.8,1.5)(7,0.8)(7.2,1.5)(7,2.2)(6.8,1.5)
\psecurve[linewidth=1.1pt, linecolor=blue](7,0.7)(7.3,1.5)(7,2.3)(6.7,1.5)(7,0.7)(7.3,1.5)(7,2.3)(6.7,1.5)
\psecurve[linewidth=1.1pt, linecolor=blue](6,1.3)(6.2,2)(6,2.7)(5.8,2)(6,1.3)(6.2,2)(6,2.7)(5.8,2)
\psecurve[linewidth=1.1pt, linecolor=blue](6,1.2)(6.3,2)(6,2.8)(5.7,2)(6,1.2)(6.3,2)(6,2.8)(5.7,2)
\psecurve[linewidth=1.1pt, linecolor=red](4.8,2)(6,1.8)(7.2,2)(6,2.2)(4.8,2)(6,1.8)(7.2,2)(6,2.2)
\psecurve[linewidth=1.1pt, linecolor=red](4.7,2)(6,1.7)(7.3,2)(6,2.3)(4.7,2)(6,1.7)(7.3,2)(6,2.3)
\psecurve[linewidth=1.1pt, linecolor=red](4.8,1.9)(5.85,2.65)(7.2,3.1)(6.15,2.35)(4.8,1.9)(5.85,2.65)(7.2,3.1)(6.15,2.35)
\psecurve[linewidth=1.1pt, linecolor=red](4.7,1.8)(5.75,2.75)(7.3,3.2)(6.25,2.25)(4.7,1.8)(5.75,2.75)(7.3,3.2)(6.25,2.25)
\psecurve[linewidth=1.1pt, linecolor=red](4.8,2.1)(6.15,1.65)(7.2,0.9)(5.85,1.35)(4.8,2.1)(6.15,1.65)(7.2,0.9)(5.85,1.35)
\psecurve[linewidth=1.1pt, linecolor=red](4.7,2.2)(6.25,1.75)(7.3,0.8)(5.75,1.25)(4.7,2.2)(6.25,1.75)(7.3,0.8)(5.75,1.25)
\psecurve[linewidth=1.1pt, linecolor=red](5,1.8)(5.2,2.5)(5,3.2)(4.8,2.5)(5,1.8)(5.2,2.5)(5,3.2)(4.8,2.5)
\psecurve[linewidth=1.1pt, linecolor=red](5,1.7)(5.3,2.5)(5,3.3)(4.7,2.5)(5,1.7)(5.3,2.5)(5,3.3)(4.7,2.5)

\psecurve[linewidth=1.1pt, linecolor=blue](5.15,2.75)(5,3.2)(4.85,2.75)(5,2.3)(5.15,2.8)(5,3.2)
\psline[linewidth=1.1pt, linecolor=blue](5.15,2.8)(7,2.8)
\psline[linewidth=1.1pt, linecolor=blue](5,3.2)(7,3.2)
\pscurve[linewidth=1.1pt, linecolor=blue](7,2.8)(7.2,3)(7,3.2)

\psecurve[linewidth=1.1pt, linecolor=blue](5.25,2.65)(5,3.3)(4.75,2.75)(5,2.2)(5.25,2.7)(5,3.3)
\psline[linewidth=1.1pt, linecolor=blue](5.25,2.7)(7,2.7)
\psline[linewidth=1.1pt, linecolor=blue](5,3.3)(7,3.3)
\pscurve[linewidth=1.1pt, linecolor=blue](7,2.7)(7.3,3)(7,3.3)

 \rput(8.5,2){\small $x$}
 \rput(9,2){$\bullet$}
 \rput(9,2.5){$\bullet$}
 \rput(9,3){$\bullet$}
 \rput(10,1.5){$\bullet$}
 \rput(10,2){$\bullet$}
 \rput(10,2.5){$\bullet$}
 \rput(11,1){$\bullet$}
 \rput(11,2){$\bullet$}
 \rput(11,3){$\bullet$}
\psecurve[linewidth=1.1pt, linecolor=red](11,0.8)(11.2,1.5)(11,2.2)(10.8,1.5)(11,0.8)(11.2,1.5)(11,2.2)(10.8,1.5)
\psecurve[linewidth=1.1pt, linecolor=red](10,1.3)(10.2,2)(10,2.7)(9.8,2)(10,1.3)(10.2,2)(10,2.7)(9.8,2)
\psecurve[linewidth=1.1pt, linecolor=blue](8.8,2)(10,1.8)(11.2,2)(10,2.2)(8.8,2)(10,1.8)(11.2,2)(10,2.2)
\psecurve[linewidth=1.1pt, linecolor=blue](8.7,2)(10,1.7)(11.3,2)(10,2.3)(8.7,2)(10,1.7)(11.3,2)(10,2.3)
\psecurve[linewidth=1.1pt, linecolor=blue](8.8,1.9)(9.85,2.65)(11.2,3.1)(10.15,2.35)(8.8,1.9)(9.85,2.65)(11.2,3.1)(10.15,2.35)
\psecurve[linewidth=1.1pt, linecolor=blue](8.8,2.1)(10.15,1.65)(11.2,0.9)(9.85,1.35)(8.8,2.1)(10.15,1.65)(11.2,0.9)(9.85,1.35)
\psecurve[linewidth=1.1pt, linecolor=blue](9,1.8)(9.2,2.5)(9,3.2)(8.8,2.5)(9,1.8)(9.2,2.5)(9,3.2)(8.8,2.5)
\psecurve[linewidth=1.1pt, linecolor=red](8.8,2.6)(9.65,2.4)(10.2,1.9)(9.35,2.1)(8.8,2.6)(9.65,2.4)(10.2,1.9)(9.35,2.1)
\psecurve[linewidth=1.1pt, linecolor=red](10.85,2.5)(11,3.2)(11.15,2.5)(11,1.8)(10.85,2.8)(11,3.2)
\psline[linewidth=1.1pt, linecolor=red](10.85,2.8)(9,2.8)
\psline[linewidth=1.1pt, linecolor=red](11,3.2)(9,3.2)
\pscurve[linewidth=1.1pt, linecolor=red](9,2.8)(8.8,3)(9,3.2)

\end{pspicture}
\end{center}
\caption{A primitive monomial walk of $\MH$}\label{Fig8}
\end{figure}} 
\end{Example}

In particular, combining Theorem~\ref{bouq_basis_classif} with Proposition~\ref{gen_no_free} we obtain the following: 

\begin{Corollary}
\label{no_free}
If the set of edges of a hypergraph $\MH$ can be partitioned into bouquets with bases then $I_{\MH}$ is strongly robust.
\end{Corollary}  

\begin{Remark}\label{2_regular}
{\em As a second application of Proposition~\ref{gen_no_free} we give a new class of hypergraphs which satisfy the conclusion of Corollary~\ref{no_free}, and whose building blocks are not necessarily bouquets with bases.  Let $\MH=(V,\ME)$ be a hypergraph such that there exists $U\subset V$ with the property that $U\cap E\neq\emptyset$ for all $E\in\ME$ and every vertex of $U$ belongs to exactly two edges. Denote by $U=\{v_1,\ldots,v_t\}\subset\{v_1,\ldots,v_m\}=V$, $\ME=\{E_1,\ldots,E_n\}$ and let $B$ be a non-free bouquet of $A_{\MH}$ (if such a $B$ does not exist then $I_{\MH}=0$ and we are done). This implies that there exists an $i\in [n]$ such that ${\pmb\alpha}_{E_i}\in B$. By definition of $\MH$ there exists $v_j\in E_i\cap E_k$ with $j\leq m$ and $k\neq i$. Then the vector ${\bf u}\in\ZZ^m$, whose only nonzero coordinate is $1$ on the j-th position satisfies the following equalities
\[
{\bf u}\cdot {\pmb\alpha}_{E_i}=1, \  {\bf u}\cdot {\pmb\alpha}_{E_k}=1, \ {\bf u}\cdot {\pmb\alpha}_{E_l}=0 \ \text{ for all } l\neq i,k.
\]    
Therefore the vector ${\bf c}_{ik}:=({\bf u}\cdot{\pmb\alpha}_{E_1},\ldots,{\bf u}\cdot{\pmb\alpha}_{E_n})$ has support $\{i,k\}$ and  $G({\pmb\alpha}_{E_i})+G({\pmb\alpha}_{E_k})={\bf 0}$, see \cite[Remark 1.3]{PTV}. Since $B$ is non-free, $G({\pmb\alpha}_{E_i})=-G({\pmb\alpha}_{E_k})\neq 0$, and thus $B$ is mixed. Therefore by Proposition~\ref{gen_no_free} we obtain the desired conclusion, and in particular we recover \cite[Proposition 4.5]{GP}. Imposing in the definition of $\MH$ that $U=V$, { thus} making $\MH$ a 2-regular hypergraph, we also recover \cite[Proposition 4.2]{GP}.}
\end{Remark}

Of course, general hypergraphs do not admit a partition of their edge sets into bouquets with bases, let alone mixed subbouquets, as it was shown earlier for $K_4$. Infinitely many such examples can be constructed; see Section~\ref{sec:complexity2} for details. 

\subsection{Sunflowers}

In this subsection we identify some interesting examples of bouquets with bases, namely, the so-called `sparse bouquets' from \cite{PeSt} (of which there was no formal definition!). These hypergraphs are built on sunflowers.  The sunflower is  highly structured and useful in the hypergraph literature; for example, it is guaranteed to occur in hypergraphs with large enough edge sets, independently of the size of the vertex set (see e.g. \cite{Juk01}.)

General properties of bouquets with bases studied in the previous subsection allow us to not only recover the theorems about existence of Graver basis elements, but also to: 1) describe completely their Graver basis elements by identifying the bouquet ideal, thus the $\ab_B$'s, 2) show that these sunflowers are actually strongly robust, and 3) identify Graver basis elements of any hypergraphs which have sunflowers as subhypergraphs. In addition, unlike \cite{PeSt}, we do not specialize to uniform hypergraphs.

In Section~\ref{sec:BouqDec} it was recalled in Theorem~\ref{all_is_well} that the bouquet graph of $A$ encodes the Graver basis of $I_A$.  On the other hand, \cite{PeSt} showed that the Graver basis of $I_\MH$ is encoded by primitive monomial walks on the hypergraph. The two concepts are consolidated in the following definition. 

\begin{Definition}\label{walk}
\rm Let $({\ME}_{blue}, {\ME}_{red})$ be a multiset collection of edges of $\MH=(V,\ME)$. We denote by $\deg_{blue}(v)$ and  $\deg_{red}(v)$  the number of edges of ${\ME}_{blue}$ and ${\ME}_{red}$  containing the vertex $v$, respectively. We say that $({\ME}_{blue}, {\ME}_{red})$ are balanced on $U\subset V$ if  $\deg_{blue}(v)=\deg_{red}(v)$ for each vertex $v\in U$. 

The vector ${\bf a}_{{\ME}_{blue}, {\ME}_{red}}=(\deg_{blue}(v)-\deg_{red}(v))_{v\in V}$ is called \emph{the vector of imbalances} of $({\ME}_{blue}, {\ME}_{red})$  and its support is contained in the complement of $U$ in $V$.  If ${\bf a}_{{\ME}_{blue}, {\ME}_{red}}={\bf 0}$ then we say that $({\ME}_{blue}, {\ME}_{red})$ is a {\em monomial walk} {\footnote{For completeness, note that the support of a monomial walk, considered as a multi-hypergraph, was called a \emph{monomial hypergraph} in \cite{PeSt}, but this definition does not make an appearance in our results. Instead, we focus on bouquets with bases and the corresponding $\MU_\ME$ and $\ME_U$.}}. 
 
\end{Definition}

Every monomial walk encodes a binomial $f_{{\ME}_B, {\ME}_R}=\prod_{E\in {\ME}_B} E - \prod_{E\in {\ME}_R} E$ in $I_{\MH}$.  A  monomial walk $({\ME}_B, {\ME}_R)$ is said to be {\em primitive} if there do not exist proper sub-multisets ${\ME'}_B\subset {\ME}_B$ and  ${\ME'}_R\subset {\ME}_R$ such that $({\ME'}_B, {\ME'}_R)$ is  also a monomial walk. The toric ideal $I_{\MH}$ is generated by binomials corresponding to primitive monomial walks, see \cite[Theorem 2.8]{PeSt}. 

\begin{Remark}\label{imbalance_bouq_basis}
{\em Let $\MH=(V,\ME)$ be a hypergraph, $U\subset V$ such that $\ME_U$ is a bouquet with basis. Since the toric ideal of the (multi)hypergraph $(U,\MU_{\ME})$ is principal generated by $e^{{\bf c}^+}-e^{{\bf c}^-}$ with $\supp({\bf c})=\ME_U$ the primitive monomial walk $(\MU_{{\ME}_{blue}},\MU_{{\ME}_{red}})$ encoded by  $(U,\MU_{\ME})$ is the following: $\MU_{{\ME}_{blue}}$ is the multiset consisting of the edges $e_1,e_{i_2},\ldots,e_{i_r}$ with multiplicities $c_1,c_{i_2},\ldots,c_{i_r}$, respectively, while $\MU_{{\ME}_{red}}$ is the multiset consisting of the edges $e_{i_{r+1}},\ldots,e_{i_t}$ with multiplicities $-c_{i_{r+1}},\ldots,-c_{i_t}$, respectively. Here $c_1,c_{i_2},\ldots,c_{i_r}$ are the positive coordinates of ${\bf c}$, while $c_{i_{r+1}},\ldots,c_{i_t}$ are the remaining coordinates of ${\bf c}$, all negative. If we consider $\ME_{blue}$ to be the multiset collection of edges $E_1,E_{i_2},\ldots,E_{i_r}$ with multiplicities $c_1,c_{i_2},\ldots,c_{i_r}$ respectively, and $\ME_{red}$ the multiset collection of edges $E_{i_{r+1}},\ldots,E_{i_t}$ with multiplicities $-c_{i_{r+1}},\ldots,-c_{i_t}$ respectively then $(\ME_{blue},\ME_{red})$ are balanced on $U$. Moreover, notice that the vector of imbalances $\ab_{\ME_{blue},\ME_{red}}$ equals $\ab_{\ME_U}$, and as explained before $\cb_{\ME_U}$ determines $(\ME_{blue},\ME_{red})$. For example, considering the bouquet with basis $\ME_{U_3}$ from Example~\ref{bouquet_basis_examples}(a), the toric ideal of $(U_3,\MU_{3_{\ME}})$ was generated by $e_{13}e_{14}e_{15}^2e_{16}-e_{17}e_{18}e_{19}e_{20}$, and thus the corresponding $\ME_{blue}$ and $\ME_{red}$ are depicted with the corresponding colors in the rightmost part of Figure~\ref{Fig8}.}
\end{Remark}

Recall that a {\em matching} on a hypergraph $\MH=(V,\ME)$ is a subset $M\subset\ME$ of pairwise disjoint edges. A matching is called {\em perfect} if it covers all the vertices of the hypergraph. A hypergraph is said to be {\em connected} if its primal graph is connected, where the primal graph has the same vertex set as the hypergraph and an edge between any two vertices contained in the same hyperedge. 

A hypergraph $\MH=(V,\ME)$ is a {\em sunflower} if, for some vertex set $C$, $E_i\cap E_j=C$  for all edges $E_i, E_j\in \ME$, $i\neq j$ and $C\subsetneq E$ for all edges $E\in\ME$. The set of vertices of $C$ is called the {\em core} of the sunflower, and each $E_i$ is called a {\em petal}. A {\em matched-petal sunflower}{\footnote{ Let us also relate this definition to \cite{PeSt}, where the authors defined a \emph{monomial sunflower}: the multi-hypergraph with ${\bf a}_{{\ME}_{blue}, {\ME}_{red}}={\bf 0}$ whose support is a matched-petal sunflower. However, there is an important distinction: not every matched-petal sunflower is the support of a monomial sunflower. (As a monomial sunflower is an example of a monomial hypergraph, this definition also isn't used in this manuscript.)}} is a hypergraph consisting of a sunflower and a perfect matching on the non-core vertices. Note that the set of edges of a matched-petal sunflower partitions into the edges of sunflower, i.e. petals, and the edges of the matching, while its set of vertices is just the set of the vertices of the sunflower. A matched-petal sunflower $\MH$ is called {\em connected} if the (multi)hypergraph $\MH-C$  is connected, where $C$ represents the core vertices. Here, $\MH-C$ is the (multi)hypergraph  consisting of the restricted sunflower: $(V\setminus C,\ME')$ where $\ME'=\{E\setminus C: E\in\ME\}$, and the edges of the perfect matching of $\MH$.   

A {\em matched-petal partitioned-core sunflower} is  a hypergraph $\MH$ consisting of a collection of vertex-disjoint sunflowers $S_1, S_2,\dots, S_l$ 
and a perfect matching on the union of non-core vertices, that is $\cup_{i}(S_i\setminus C_i)$. A matched-petal partitioned-core 
sunflower is called {\em connected} if $\MH-C$ is connected, where $C=\cup_{i}C_i$. A {\em matched-petal relaxed-core sunflower} is a hypergraph $\MH$ consisting of a collection of sunflowers $S_1,S_2,\ldots,S_l$ with cores $C_1,\ldots,C_l$ respectively, which may only intersect at their cores, and a perfect matching on the union of the non-core vertices of the sunflowers. A matched-petal relaxed-core sunflower is called {\em connected} if the (multi)hypergraph $\MH-C$ is connected, where $C=\cup_{i}C_i$. 

\begin{Remark}\label{explaining_sunflowers}
{\em 
The definitions above resemble various monomial walks (based on  various types of sunflowers) introduced in \cite{PeSt}. In contrast, here we allow hypergraphs to be non-uniform and consider the supporting sunflowers as sets instead of multisets. 

As far as the standard terminology is concerned, note that a matched-petal sunflower is a particular case of matched-petal partitioned-core sunflower, which in turn is a particular case of matched-petal relaxed-core sunflower. Each of the first two subhypergraphs from Figure~\ref{three_bouq} are connected matched-petal sunflowers, while the third one is not. If we identify the vertices $x_1$, $x_2$ and $x_3$ from Figure~\ref{sunflower_curve} then the hypergraph consisting of the three depicted connected matched-petal sunflowers is a non-connected matched-petal sunflower. }
\end{Remark}

\begin{Theorem}\label{connected_sunflower} 
Every connected matched-petal relaxed-core sunflower is a bouquet with basis, where the basis consists of the non-core vertices. 
\end{Theorem}

\begin{proof}
Let $\MH$ be a connected matched-petal relaxed-core sunflower consisting of $l$ sunflowers $S_1,\ldots,S_l$ with $C_i$ being the core vertices of sunflower $S_i$ for all $i$, and let $U=\cup_{i=1}^l S_i\setminus C_i$. Note that by definition the set of edges $\ME$ of $\MH$ partitions into the set of petals of the sunflowers $S_1,\ldots,S_l$, labeled $E_1,\ldots,E_t$, and the edges of the matching, labeled $E_{t+1},\ldots,E_k$. In order to prove that $\ME_U$ is a bouquet with basis, we note first that the vector ${\bf c}=(1,1,\ldots,1,-1,\ldots,-1)\in\ZZ^k$ corresponds to a binomial in the toric ideal of $(U,\MU_{\ME})$, where the number of $1$'s equals the number of petals, while the number of $-1$'s equals the number of edges of the matching. Assume that ${\bf u}=(u_1,\ldots,u_t,u_{t+1},\ldots,u_k)$ is an arbitrary vector corresponding to a binomial in the toric ideal of $(U,\MU_{\ME})$. It remains to prove that $u_1=\cdots=u_t$ and $u_{t+1}=\cdots=u_k=-u_1$. Let $E,E'\in\MU_{\ME}$ be two different edges, restrictions of two petals and thus non-empty, and let $v\in E$, $v'\in E'$. 

Since $\MH$ is connected, there exists a path in the primal graph of $(U,\MU_{\ME})$ from $v$ to $v'$, that is $v=v_0,v_1,\ldots,v_r=v'$. We construct inductively a sequence of edges of $\MU_{\ME}$: let $s_1$ be the largest number such that $v_{s_1}\in E:=E_1$. Since $E,E'$ are restrictions of petals then $E\cap E'=\emptyset$, and thus $s_1<r$. By definition of $s_1$, there exists an edge of $\MU_{\ME}$, say $E_2$, such that $v_{s_1}\in E_1\cap E_2$. Let $s_2$ be the largest number such that $v_{s_2}\in E_2$. If $s_2=r$ then $E'=E_2$ and we stop, otherwise we continue. In this way we obtain a sequence of edges $E=E_1,\ldots,E_{p}=E'$ of $\MU_{\ME}$ for some $p\geq 2$. 

Denote by $u_{i_j}$ the coordinate of $u$ corresponding to $E_j$ for all $j=1,\ldots,p$ (that is $|u_{i_j}|$ is the exponent of $E_j$ in the binomial). Since the binomial corresponding to $u$ is in the toric ideal of $(U,\MU_{\ME})$, every vertex of $U$ is balanced. Moreover, since every vertex of $U$ belongs to exactly two edges, $v_{s_j}$ being balanced implies that $u_{i_{j+1}}=-u_{i_j}$ for all $j<p$. Thus we obtain $u_{i_j}=(-1)^{j+1}u_1$ for all $j=1,\ldots,p$, and in particular, if $E_j$ is the restriction of a petal then $E_{j+1}$ is an edge of a matching. Therefore, we obtain that for any distinct edges $E,E'$ of $\MU_{\ME}$ the corresponding coordinates of ${\bf u}$ are either equal or negatives of each other, with equality holding if and only if $E,E'$ are simultaneously either restrictions of petals or edges of the matching $M$. Hence we get that ${\bf u}=(u_1,\ldots,u_1,-u_1,\ldots,-u_1)$, which implies that the toric ideal of $(U,\MU_{\ME})$ is principal, and generated by $e^{{\bf c}^+}-e^{{\bf c}^-}$. Thus $\ME_U$ is a bouquet with basis, as desired. \qed

\end{proof} 

Since the connected components of a matched-petal relaxed-core sunflower are connected matched-petal relaxed-core sunflowers, then by  Theorem~\ref{connected_sunflower} and Corollary~\ref{no_free} we obtain the following: 

\begin{Proposition}\label{sunflowers}
Let $\MH$ be a matched-petal relaxed-core sunflower. Then $I_{\MH}$ is strongly robust.
\end{Proposition}

In particular, one recovers \cite[Theorem 4.12]{PeSt} and implicitly \cite[Proposition 4.5]{PeSt} and \cite[Proposition 4.9]{PeSt}. To see this, we first identify the subbouquets of a matched-petal relaxed-core sunflower $\MH$ and their corresponding $\ab$-vectors. If $\MH_1,\ldots,\MH_t$ are the connected components of $\MH$, that is matched-petal relaxed-core sunflowers with the sets of core vertices $C_1,\ldots,C_t$, then denote by $C=\{v_1,\ldots,v_s\}$ the union $\cup_{i}C_i$ of core vertices of $\MH$. Moreover, we label the edges such that petals are labeled first, while the edges of the matching are labeled last. By Theorem~\ref{connected_sunflower}, $\MH_1,\ldots,\MH_t$ are bouquets with bases, the bases are the sets of non-core vertices, and the vectors $\ab_{\MH_i}$ can be computed as vectors of imbalances induced by $\cb_{\MH_i}$, as explained in Remark~\ref{imbalance_bouq_basis}. Therefore, if we label the rows of the incidence matrix of $\MH$ first according to the vertices from $C$, then it follows from Remark~\ref{imbalance_bouq_basis} that for each $j$ the vector $\ab_{\MH_j}$ has at most the first $s$ components nonzero and for each $i=1,\ldots,s$ the $i$-th coordinate of $\ab_{\MH_j}$ is equal to $d_{ij}$, where $d_{ij}$ is the number of petals of $\MH_j$ containing the vertex $v_i$. Now it is obvious via Theorem~\ref{all_is_well} that $I_{\MH}\neq 0$ if and only if $I_{D}\neq 0$, where $D$ is the matrix $(d_{ij})\in\ZZ^{s\times t}$, and their Graver basis are in bijective correspondence,  and this is essentially the content of \cite[Theorem 4.12]{PeSt}.

\begin{Example}\label{matched-petal}
{\em Consider $\MH$ to be the matched-petal relaxed-core sunflower whose three connected components $\MH_1,\MH_2,\MH_3$ are depicted below, and with cores $C_i=\{x_i\}$ for $i=1,2,3$. Note that each $\MH_i$ is a connected matched-petal sunflower, and thus a bouquet with basis by Theorem~\ref{connected_sunflower}. If $x_i\neq x_j$ for all $i\neq j$ then the union of core vertices is $C=\{x_1,x_2,x_3\}$ and the matrix $(d_{ij})\in\ZZ^{3\times 3}$ is the displayed matrix $D_1$. We recall from the previous paragraph that $d_{ij}$ is the number of petals of $\MH_j$ containing $x_i$. Thus in this case $I_{\MH}=0$.       
\[
D_1=\left( \begin{array}{ccc}
3 & 0 & 0 \\
0 & 4 & 0 \\
0 & 0 & 5  
\end{array} \right) \hspace{0.5cm} D_2=\left( \begin{array}{ccc}
3 & 4 & 0 \\
0 & 0 & 5  
\end{array} \right) \hspace{0.5cm} D_3=\left( \begin{array}{ccc}
3 & 4 & 5  
\end{array} \right) 
\]
If $x_1=x_2\neq x_3$ then the union of core vertices is $C=\{x_1,x_3\}$ and we obtain the corresponding matrix $D_2\in\ZZ^{2\times 3}$. Therefore in this case
$I_{\MH}\neq 0$ is principal. Finally, if $x_1=x_2=x_3$ then $C=\{x_1\}$ and the matrix $D_3$ is given above. 

\begin{figure}[hbt]
\label{sunflower_curve}
\begin{center}
\psset{unit=0.9cm}
\begin{pspicture}(0,0.75)(12,3.5)
 \rput(1,1){$\bullet$}
 \rput(1,2){$\bullet$}
 \rput(1,3){$\bullet$}
 \rput(2,1.5){$\bullet$}
 \rput(2,2){$\bullet$}
 \rput(2,2.5){$\bullet$}
 \rput(3,2){$\bullet$}
 \rput(3.6,2){\small $x_1$}
\psecurve[linewidth=1.1pt, linecolor=black](1,0.8)(1.35,2)(1,3.2)(0.65,2)(1,0.8)(1.35,2)(1,3.2)(0.65,2)
\psecurve[linewidth=1.1pt, linecolor=black](2,1.3)(2.2,2)(2,2.7)(1.8,2)(2,1.3)(2.2,2)(2,2.7)(1.8,2)
\psecurve[linewidth=1.1pt, linecolor=black](3.2,2)(2,2.2)(0.8,2)(2,1.8)(3.2,2)(2,2.2)(0.8,2)(2,1.8)
\psecurve[linewidth=1.1pt, linecolor=black](3.2,1.9)(2.15,2.65)(0.8,3.1)(1.85,2.35)(3.2,1.9)(2.15,2.65)(0.8,3.1)(1.85,2.35)
\psecurve[linewidth=1.1pt, linecolor=black](3.2,2.1)(1.85,1.65)(0.8,0.9)(2.15,1.35)(3.2,2.1)(1.85,1.65)(0.8,0.9)(2.15,1.35)
 
 \rput(4.5,2){\small $x_2$}
 \rput(5,2){$\bullet$}
 \rput(5,2.5){$\bullet$}
 \rput(5,3){$\bullet$}
 \rput(6,1.5){$\bullet$}
 \rput(6,2){$\bullet$}
 \rput(6,2.5){$\bullet$}
 \rput(7,1){$\bullet$}
 \rput(7,2){$\bullet$}
 \rput(7,3){$\bullet$}
\psecurve[linewidth=1.1pt, linecolor=black](7,0.8)(7.2,1.5)(7,2.2)(6.8,1.5)(7,0.8)(7.2,1.5)(7,2.2)(6.8,1.5)
\psecurve[linewidth=1.1pt, linecolor=black](6,1.3)(6.2,2)(6,2.7)(5.8,2)(6,1.3)(6.2,2)(6,2.7)(5.8,2)
\psecurve[linewidth=1.1pt, linecolor=black](4.8,2)(6,1.8)(7.2,2)(6,2.2)(4.8,2)(6,1.8)(7.2,2)(6,2.2)
\psecurve[linewidth=1.1pt, linecolor=black](4.8,1.9)(5.85,2.65)(7.2,3.1)(6.15,2.35)(4.8,1.9)(5.85,2.65)(7.2,3.1)(6.15,2.35)
\psecurve[linewidth=1.1pt, linecolor=black](4.8,2.1)(6.15,1.65)(7.2,0.9)(5.85,1.35)(4.8,2.1)(6.15,1.65)(7.2,0.9)(5.85,1.35)
\psecurve[linewidth=1.1pt, linecolor=black](5,1.8)(5.2,2.5)(5,3.2)(4.8,2.5)(5,1.8)(5.2,2.5)(5,3.2)(4.8,2.5)
\psecurve[linewidth=1.1pt, linecolor=black](5.15,2.75)(5,3.2)(4.85,2.75)(5,2.3)(5.15,2.8)(5,3.2)
\psline[linewidth=1.1pt, linecolor=black](5.15,2.8)(7,2.8)
\psline[linewidth=1.1pt, linecolor=black](5,3.2)(7,3.2)
\pscurve[linewidth=1.1pt, linecolor=black](7,2.8)(7.2,3)(7,3.2)

 \rput(8.5,2){\small $x_3$}
 \rput(9,2){$\bullet$}
 \rput(9,2.5){$\bullet$}
 \rput(9,3){$\bullet$}
 \rput(9,1.5){$\bullet$}
 \rput(9,1){$\bullet$}
 \rput(10,1.5){$\bullet$}
 \rput(10,2){$\bullet$}
 \rput(10,2.5){$\bullet$}
 \rput(11,1){$\bullet$}
 \rput(11,2){$\bullet$}
 \rput(11,3){$\bullet$}
\psecurve[linewidth=1.1pt, linecolor=black](10,1.3)(10.2,2)(10,2.7)(9.8,2)(10,1.3)(10.2,2)(10,2.7)(9.8,2)
\psecurve[linewidth=1.1pt, linecolor=black](9,0.8)(9.2,1.5)(9,2.2)(8.8,1.5)(9,0.8)(9.2,1.5)(9,2.2)(8.8,1.5)
\psecurve[linewidth=1.1pt, linecolor=black](8.75,2)(10,1.8)(11.25,2)(10,2.2)(8.75,2)(10,1.8)(11.25,2)(10,2.2)
\psecurve[linewidth=1.1pt, linecolor=black](8.8,1.9)(9.85,2.65)(11.2,3.1)(10.15,2.35)(8.8,1.9)(9.85,2.65)(11.2,3.1)(10.15,2.35)
\psecurve[linewidth=1.1pt, linecolor=black](8.8,2.1)(10.15,1.65)(11.2,0.9)(9.85,1.35)(8.8,2.1)(10.15,1.65)(11.2,0.9)(9.85,1.35)
\psecurve[linewidth=1.1pt, linecolor=black](9,1.8)(9.2,2.5)(9,3.2)(8.8,2.5)(9,1.8)(9.2,2.5)(9,3.2)(8.8,2.5)
\psecurve[linewidth=1.1pt, linecolor=black](10.85,2.5)(11,3.2)(11.15,2.5)(11,1.8)(10.85,2.8)(11,3.2)
\psline[linewidth=1.1pt, linecolor=black](10.85,2.8)(9,2.8)
\psline[linewidth=1.1pt, linecolor=black](11,3.2)(9,3.2)
\pscurve[linewidth=1.1pt, linecolor=black](9,2.8)(8.8,3)(9,3.2)
\psccurve[linewidth=1.1pt, linecolor=black](9,2.32)(9.18,2.5)(9,2.68)(8.82,2.5)(9,2.32)(9.18,2.5)(9,2.68)(8.82,2.5)
\psccurve[linewidth=1.1pt, linecolor=black](9,1.32)(9.18,1.5)(9,1.68)(8.82,1.5)(9,1.32)(9.18,1.5)(9,1.68)(8.82,1.5)
\psecurve[linewidth=1.1pt, linecolor=black](8.75,1)(10,0.8)(11.25,1)(10,1.2)(8.75,1)(10,0.8)(11.25,1)(10,1.2)

\end{pspicture}
\end{center}
\caption{Matched-petal sunflowers}\label{sunflower_curve}
\end{figure}
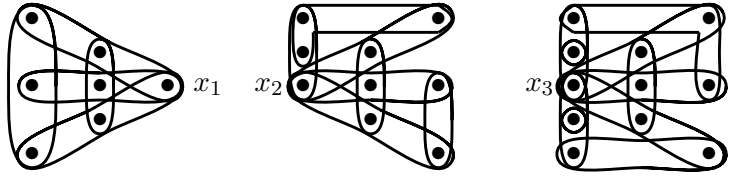
}
\end{Example}

Specializing the previous discussions to a matched-petal sunflower $\MH$ with $t\geq 1$ connected components one obtains the following classification. The toric ideal $I_{\MH}=0$ if and only if $t=1$, that is $\MH$ is a connected matched-petal sunflower. The toric ideal $I_{\MH}\neq 0$ if and only if $t>1$, in which case the subbouquet ideal is just the toric ideal of the monomial curve $(d_1 \ldots d_t)$, where $d_i$ represents the number of petals of the $i$-th connected component, containing the core.    

\subsection{Graver basis elements of hypergraphs}

 In cases of interest, it may happen that all (or almost all) bouquets are singletons, seemingly implying that the bouquet construction does not offer anything. However, this is not the case since by \cite[Proposition 4.13]{St} if we restrict to a submatrix $C$ of $A$ obtained 
by deleting some of the columns of $A$ then the Graver basis, universal Gr\"obner basis and circuits of $C$ are included in the corresponding ones of $A$. Based on this and the bouquet techniques we construct in Proposition~\ref{graver_not_ugb} an element in the Graver basis of the uniform complete hypergraph with large enough number of vertices, which is not in the universal Gr\"obner basis. 

In \cite{DST, TT} it was shown that the  universal Gr{\"o}bner basis and the Graver basis of the toric ideal of the complete graph $K_n$ are identical for $n\leq 8$ and differ for $n\geq 9$. As one application of the bouquet technique, we show that the universal Gr{\"o}bner basis and the Graver basis of the toric ideal of the complete d-uniform hypergraph $K^d_n$ differ for $n\geq (d+1)^2$, by giving a non trivial example of an element in the Graver basis which does not belong to the universal Gr{\"o}bner basis.   First note that all of the bouquets of the complete d-uniform hypergraph $K^d_n$, for large $n$, are singletons. Then, restricting to a subhypergraph $\MH$ of $K^d_n$, as explained before, we can apply the bouquet techniques to find an element from the Graver basis of $\MH$ that belongs also to the Graver basis of $K^d_n$. 

We consider a hypergraph $\MH_{d+1}=(V,\ME)$ with $(d+1)^2$ vertices $V=\{ v_{ij}| 1\leq i,j \leq d+1\}$ and $(d+2)(d+1)$ edges: let $E_j=\{v_{k1}|1\leq k\leq d+1,
k\not=j\}$  for $1\leq j\leq d+1$ and $E_{ij}=\{v_{ik}|1\leq k\leq d+1, k\not=j\}$ for $1\leq i, j\leq d+1$. The hypergraph $\MH_4$, for $d=3$, is depicted in Figure~\ref{5tetra}.

Note that $\MH_{d+1}$ is a subhypergraph of $K^d_n$ for $n\geq (d+1)^2$. In addition the set of edges of $\MH_{d+1}$ partitions into $d+1$ bouquets with bases $\ME_{U_1},\ldots,\ME_{U_{d+1}}$ and $d+1$ single edges $E_1,\ldots,E_{d+1}$. Here $\ME_{U_i}=\{E_{ij}| 1\leq j\leq d+1\}$ and $U_i=\{v_{ij}| 2\leq j\leq d+1\}$ for all $i=1,\ldots,d+1$, while the principal generator of the toric ideal of the hypergraph $(U_i,\MU_{i_{\ME}})$ is the binomial $\prod_{j\neq 1}e_{ij}-e_{i1}^{d-1}$, see Definition~\ref{defi_bouquet}. Moreover, the incidence matrix of $\MH_{d+1}$ has the rows indexed by the vertices in the following way: $v_{11},v_{21},\ldots,{v_{d+1,1}},v_{12},v_{22},\ldots,v_{d+1,d+1}$ and the columns indexed by the edges in the following way $E_{11},E_{12},\ldots,E_{1,d+1},E_{21},E_{22},\ldots,E_{d+1,d+1},E_1,\ldots,E_{d+1}$. With respect to this labeling of the edges and similarly to the computations from Example~\ref{bouquet_basis_examples}(b) we obtain that $\cb_{\ME_{U_1}}=(d-1,-1,\ldots,-1,{\bf 0},\ldots,{\bf 0})$, $\cb_{\ME_{U_2}}=({\bf 0},d-1,-1,\ldots,-1,{\bf 0},\ldots,{\bf 0})$, and so on $\cb_{\ME_{U_{d+1}}}=({\bf 0},\ldots,{\bf 0},d-1,-1,\ldots,-1,{\bf 0})$ are vectors of $\ZZ^{(d+2)(d+1)}$, where ${\bf 0}\in\ZZ^{d+1}$ represents the vector with all coordinates zero. By Remark~\ref{imbalance_bouq_basis} the vector $\ab_{\ME_{U_i}}\in\ZZ^{(d+1)^2}$ has all coordinates zero except the one corresponding to the vertex $v_{i1}$, which equals $-d$.  For the singleton subbouquets $E_1,\ldots,E_{d+1}$, the encoding vectors are: $\cb_{E_1}=({\bf 0},\ldots,{\bf 0},1,0,\ldots,0)$, $\dots$,  $\cb_{E_{d+1}}=({\bf 0},\ldots,{\bf 0},0,\ldots,0,1)${\bf,} while $\ab_{E_i}={\pmb\alpha}_{E_i}$ for all $i=1,\ldots,d+1$. 

\begin{figure}[hbt]
\label{5tetra}
\begin{center}
\psset{unit=0.5cm}
\begin{pspicture}(3,1)(10,10)

\pspolygon[fillcolor=medium,fillstyle=solid](4,7.25)(5.5,5.7)(4,4.15)(3,5.7)
\psline(4,7.25)(4,4.15)
\psline[linestyle=dashed](3,5.7)(5.5,5.7)
 
\pspolygon[fillcolor=medium,fillstyle=solid](7,4.3)(5.5,5.7)(7,7.4)(8.5,5.7)
\psline[linestyle=dashed](7,7.4)(7,4.3)
\psline(5.5,5.7)(8.5,5.7) 
 
\pspolygon[fillcolor=medium,fillstyle=solid](7,4.3)(5.5,2.4)(7,1.2)(8.5,2.4)
\psline(7,1.2)(7,4.3)
\psline[linestyle=dashed](5.5,2.4)(8.5,2.4) 
 
\pspolygon[fillcolor=medium,fillstyle=solid](10,4.15)(11,5.7)(10,7.25)(8.5,5.7)
\psline(10,7.25)(10,4.15)
\psline[linestyle=dashed](8.5,5.7)(11,5.7) 
 
\pspolygon[fillcolor=medium,fillstyle=solid](7,10.5)(5.5,9.3)(7,7.4)(8.5,9.3)
\psline(7,7.4)(7,10.5)
\psline[linestyle=dashed](5.5,9.3)(8.5,9.3) 
 
\rput(6.3,7.4){\tiny $v_{11}$} \rput(5.5,6.4){\tiny $v_{21}$}  \rput(7.7,4.3){\tiny $v_{31}$}  \rput(8.5,5){\tiny $v_{41}$} 
\rput(5,9.3){\tiny $v_{12}$} \rput(7,10.8){\tiny $v_{13}$} \rput(9,9.3){\tiny $v_{14}$} 
\rput(2.5,5.7){\tiny $v_{23}$} \rput(4,3.85){\tiny $v_{22}$} \rput(4,7.55){\tiny $v_{24}$} 
\rput(5,2.4){\tiny $v_{32}$} \rput(7,0.9){\tiny $v_{33}$} \rput(9,2.4){\tiny $v_{34}$} 
\rput(10,3.85){\tiny $v_{42}$} \rput(11.5,5.7){\tiny $v_{43}$} \rput(10,7.55){\tiny $v_{44}$}

  \end{pspicture}
\end{center}
\caption{$\MH_{4}$}\label{5tetra}
\end{figure}

For the rest of this subsection we use for simplicity the binomial representation for an element in the toric ideal of a hypergraph instead of the vector representation. It is easy to see that the following three types of binomials belong to the toric ideal of the complete d-uniform hypergraph $K^d_n$: a) $d+1$ binomials of the form $$g_i:=E_i^{d-1}\prod_{j\not=1}E_{ij}-E_{i1}^{d-1}\prod_{j\not=i}E_{j},$$ b) $d+1$ binomials of the form
$$h_i:=\prod_{i\not=l}\prod_{j\not=1}E_{ij}-E_{l}^{d}\prod_{i\not=l}E_{i1}^{d-1},$$
and c) one binomial of the form
$$\prod_{i}\prod_{j\not=1}E_{ij}-\prod_{i}E_{i}\prod_{i}E_{i1}^{d-1}.$$
Furthermore, it can be shown that they are all elements of the Graver basis of $K^d_n$, although we will prove next just for the last binomial.

\begin{Proposition}\label{graver_not_ugb}
The universal Gr{\"o}bner basis of the toric ideal of the complete $d$-uniform hypergraph $K^d_n$ differs from the Graver basis for $n\geq (d+1)^2$.
\end{Proposition}
\begin{proof} 
First we prove that the binomial $$f=\prod_{i}\prod_{j\not=1}E_{ij}-\prod_{i}E_{i}\prod_{i}E_{i1}^{d-1}$$ does not belong to the universal Gr{\"o}bner basis of the toric ideal of the complete $d$-uniform hypergraph $K^d_n$. We argue by contradiction. Assume that there exists a monomial order $>$ such that the binomial $f$ belongs to the reduced Gr{\"o}bner basis with respect to the order $>$. Since $E_{i1}^{d-1}\prod_{j\not=i}E_{j}$ divides properly   $\prod_{i}E_{i}\prod_{i}E_{i1}^{d-1}$, it follows that
\begin{eqnarray}\label{1category}
E_i^{d-1}\prod_{j\not=1}E_{ij}>E_{i1}^{d-1}\prod_{j\not=i}E_{j}.
\end{eqnarray}
Indeed, if this is not the case then $\ini_<(g_i)=E_{i1}^{d-1}\prod_{j\not=i}E_{j}$. From the previous divisibility it follows that: 1) if  $\ini_<(f)=\prod_{i}E_{i}\prod_{i}E_{i1}^{d-1}$ then we contradict that $f$ is in the Gr\"obner basis, or 2) if $\ini_<(f)=\prod_{i}\prod_{j\not=1}E_{ij}$ then we contradict that $f$ is reduced. Thus we obtain the inequality (\ref{1category}).

Also $\prod_{i\not=l}\prod_{j\not=1}E_{ij}$ divides $\prod_{i}\prod_{j\not=1}E_{ij}$, and similarly to the proof of (\ref{1category}) it can be shown that  
\begin{eqnarray}\label{2category}
\prod_{i\not=l}\prod_{j\not=1}E_{ij}<E_{l}^{d}\prod_{i\not=l}E_{i1}^{d-1}.
\end{eqnarray}
Taking the product of inequalities (\ref{1category}) when $i$ runs from $1$ to $d+1$ and canceling common terms we obtain
$$\prod_{i}\prod_{j\not=1}E_{ij}>\prod_{i}E_{i}\prod_{i}E_{i1}^{d-1}.$$
Similarly, taking the product of inequalities (\ref{2category}) we obtain
$$(\prod_{i}\prod_{j\not=1}E_{ij})^d<(\prod_{i}E_{i}\prod_{i}E_{i1}^{d-1})^d,$$
a contradiction.

In order to prove that the binomial $f$ belongs to the Graver basis of $K^d_n$ it is enough to prove via \cite[Proposition 4.13]{St} that $f$ belongs to the Graver basis of its subhypergraph $\MH_{d+1}$. The partition of the edges of $\MH_{d+1}$ into $d+1$ bouquets with bases $\ME_{U_1},\ldots,\ME_{U_{d+1}}$ and $d+1$ single edges $E_1,\ldots,E_{d+1}$ induces the matrix $A_B\in\ZZ^{(d+1)^2\times 2(d+1)}$
\[
A_B=[\ab_{\ME_{U_1}},\ldots,\ab_{\ME_{U_{d+1}}},\ab_{E_1},\ldots,\ab_{E_{d+1}}]=\begin{array}{c}
\left( \begin{array}{cccccccc}
-d & 0 & \ldots & 0 & 0 & 1 & \ldots & 1\\
0 & -d & \ldots & 0 & 1 & 0 & \ldots & 1\\
  &  & \ddots &  & & & \ddots & \\
0 & 0 & \ldots & -d & 1 & 1 & \ldots & 0\\
{\bf 0} & {\bf 0} & \ldots & {\bf 0} & {\bf 0} & {\bf 0} & \ldots & {\bf 0}\\
\end{array} \right)\end{array} \ ,
\]
where $\ab_{E_i}={\pmb\alpha}_{E_i}$ for all $i$, the first $d+1$ rows are indexed after the vertices $v_{11},\ldots,v_{d+1,1}$ and the last row represents block matrix ${\bf 0}\in\ZZ^{d(d+1)\times 2(d+1)}$. As it was noticed in the comments prior to this proposition the binomial $f$ corresponds to the vector $${\bf v}=(1-d,1,\ldots,1,1-d,1,\ldots,1,\ldots,1-d,1,\ldots,1,-1,-1,\ldots,-1)\in\Ker(A_{\MH_{d+1}}).$$ 
Applying Theorem~\ref{all_is_well} we have a bijective correspondence between $\Ker(A_B)$ and $\Ker(A_{\MH_{d+1}})$ and given by $$B(u_1,\ldots,u_{2d+2})=\sum_{i=1}^{d+1}{\cb_{\ME_{U_i}}}u_i+\sum_{i=1}^{d+1}{\cb_{E_i}}u_{d+1+i}.$$
Therefore replacing the formulas obtained before for $\cb_{\ME_{U_i}}$ and $\cb_{E_i}$, for all $i=1,\ldots,d+1$, we obtain that ${\bf v}=B({\bf u})$ where ${\bf u}=(1,\ldots,1,-1,\ldots,-1)\in\ZZ^{2d+2}$ with equally many $1$'s and $-1$'s. By Theorem~\ref{all_is_well} it is enough to prove that ${\bf u}$ belongs to the Graver basis of $A_B$ to conclude that ${\bf v}$ (and thus $f$) is in the Graver basis of $\MH_{d+1}$. Assume by contradiction ${\bf u}={\bf u}^+-{\bf u}^-$ is not in the Graver basis of $A_B$. Then there exists ${\bf 0}\neq {\bf w}\in\Ker(A_B)$ such that ${\bf w}^+\leq{\bf u}^+$ and ${\bf w}^-\leq{\bf u}^-$, with at least one inequality strict. In terms of coordinates this can be restated: $w_i\in\{0,1\}$ for all $i=1,\ldots,d+1$ and $w_i\in\{-1,0\}$ for all $i=d+2,\ldots,2d+2$, and with at least one coordinate zero. If ${\bf w}^+={\bf u}^+$ then since ${\bf w}\in\Ker(A_B)$ we obtain ${\bf w}^-={\bf u}^-$, a contradiction. Otherwise there exists $i\in \{1,\ldots,d+1\}$ such that $w_i=0$, and without loss of generality we may assume that $w_1=0$. Since ${\bf w}\in\Ker(A_B)$ then $w_{d+3}+\cdots+w_{2d+2}=0$, and thus $w_{d+3}=\ldots=w_{2d+2}=0$. If $w_{d+2}=0$ then ${\bf w}={\bf 0}$, a contradiction, so we necessarily have $w_{d+2}=-1$. This leads to $-dw_i=1$ for all $i=1,\ldots,d+1$ a contradiction to ${\bf w}\in\ZZ^{2d+2}$. Therefore ${\bf u}$ belongs to the Graver basis of $A_B$, and we are done.    \qed

\end{proof}

\section{Complexity of hypergraphs}
\label{sec:complexity}

It is well-known that the elements of the Graver basis of a toric ideal of a graph are of a rather special form: the exponents of each variable in an element of the Graver basis is either one or two, see, for example, \cite{RTT} and references therein, and are completely determined by their support. In other words, one can not find two elements in the Graver basis of a toric ideal of a graph with the same support. In \cite{PeSt} it was shown that exponents in the elements of the Graver basis of a toric ideal of a hypergraph can be arbitrarily high, and are not uniquely determined by their supports, see Example~\ref{graver_bouquets_bases}. Thus the natural question is: how complicated are the Graver basis, the universal Gr{\"o}bner basis  or a Markov basis of a hypergraph? 

The bouquet ideal technique can be used to prove that toric ideals of hypergraphs are as complicated as toric ideals in general. Namely, for any toric ideal, there exists a toric ideal of a hypergraph with ``worse" Graver basis, universal Gr{\"o}bner basis and Markov bases, as well as circuits. ``Worse" means that the corresponding set for the toric ideal of a hypergraph has at least the same cardinality and elements of higher degrees than the corresponding elements of the toric ideal. In \cite[Theorem 2.1]{DO} the authors show that for a particular matrix $A$, namely the one corresponding to the no-three-way interaction model in statistics, elements in a minimal Markov basis can be arbitrarily complicated as two of the three dimensions of the underlying table grow to infinity. Such matrices are, in fact, incidence matrices of $3$-uniform hypergraphs. In contrast, Theorem~\ref{worse} implies that the complexity of the whole Graver basis (resp. minimal Markov or Universal Gr{\" o}bner basis) of  any matrix $A$  can be captured by an almost $3$-uniform hypergraph, that is a hypergraph whose edges have cardinality at most $3$. In Section~\ref{sec:complexity2} there will be a stronger connection but for matrices $A$ defining positively graded toric ideals.

We start first with an example, which shows the details of the general construction.

\begin{Example}\label{toric->hyper}
{\em Let 
\[
A=\left( \begin{array}{cccc}
-1 & -1 & 2 & 2\\
-2 & 2 & -1 & 0  
\end{array} \right),
\]
be the matrix whose columns are denoted by ${\bf a}_1,\ldots,{\bf a}_4$. We construct a hypergraph $\MH=(V,\ME)$ such that the bouquet graph of $\MH$ has four bouquets with bases $\ME_{U_1},\ldots,\ME_{U_4}$, and with the property that the corresponding subbouquet ideal (associated to the vectors ${\bf a}_{\ME_{U_1}},\ldots,{\bf a}_{\ME_{U_4}}$) is $I_A$. The latter will follow since the vectors ${\bf a}_{\ME_{U_1}},\ldots,{\bf a}_{\ME_{U_4}}$ are ``essentially" the vectors ${\bf a}_1,\ldots,{\bf a}_4$, in the sense that each ${\bf a}_{\ME_{U_i}}\in\ZZ^{|V|}$ is just the natural embedding of ${\bf a}_i$ in $\ZZ^{|V|}$ with the other coordinates $0$. The construction of $\MH$ is carried out in three steps. 

{\bf Step 1.} Every non-zero entry of the matrix $A$ is used to construct a sunflower, which will be the building blocks of the desired hypergraph. Precisely, for a positive entry $\lambda$ of the matrix $A$ we consider a $3$-uniform sunflower with one core vertex and $|\lambda|$ petals, while for a negative entry $\lambda$ we consider the sunflower with one core vertex and $|\lambda|+1$ petals, which is almost 3-uniform, meaning that only one petal has two vertices, the other have three. In the particular case of our example the matrix $A$ has three different nonzero entries $-1,-2,2$, and thus we have two almost $3$-uniform sunflowers and one $3$-uniform sunflower pictured below, see Figure~\ref{sunflowers}.

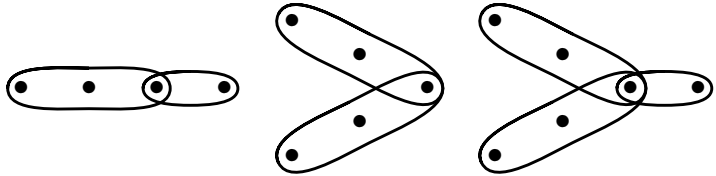
\begin{figure}[hbt]
\label{sunflowers}
\begin{center}
\psset{unit=0.9cm}
\begin{pspicture}(0,0.75)(12,3.5)
 
 \rput(1,2){$\bullet$}
  \rput(2,2){$\bullet$}
 \rput(4,2){$\bullet$}
 \rput(3,2){$\bullet$}
 
\psecurve[linewidth=1.1pt, linecolor=black](4.2,2)(3.5,2.25)(2.8,2)(3.5,1.75)(4.2,2)(3.5,2.25)(2.8,2)(3.5,1.75)
\psecurve[linewidth=1.1pt, linecolor=black](3.2,2)(2,2.3)(0.8,2)(2,1.7)(3.2,2)(2,2.3)(0.8,2)(2,1.7)

 \rput(5,1){$\bullet$}
 \rput(5,3){$\bullet$}
 \rput(6,1.5){$\bullet$}
 \rput(6,2.5){$\bullet$}
 \rput(7,2){$\bullet$}
 
\psecurve[linewidth=1.1pt, linecolor=black](7.2,1.9)(6.15,2.8)(4.8,3.1)(5.85,2.2)(7.2,1.9)(6.15,2.8)(4.8,3.1)(5.85,2.2)
\psecurve[linewidth=1.1pt, linecolor=black](7.2,2.1)(5.85,1.8)(4.8,0.9)(6.15,1.2)(7.2,2.1)(5.85,1.8)(4.8,0.9)(6.15,1.2)

 \rput(8,1){$\bullet$}
 \rput(8,3){$\bullet$}
 \rput(9,1.5){$\bullet$}
 \rput(9,2.5){$\bullet$}
 \rput(10,2){$\bullet$}
 \rput(11,2){$\bullet$}
 
\psecurve[linewidth=1.1pt, linecolor=black](10.2,1.9)(9.15,2.8)(7.8,3.1)(8.85,2.2)(10.2,1.9)(9.15,2.8)(7.8,3.1)(8.85,2.2)
\psecurve[linewidth=1.1pt, linecolor=black](10.2,2.1)(8.85,1.8)(7.8,0.9)(9.15,1.2)(10.2,2.1)(8.85,1.8)(7.8,0.9)(9.15,1.2)
\psecurve[linewidth=1.1pt, linecolor=black](11.2,2)(10.5,2.25)(9.8,2)(10.5,1.75)(11.2,2)(10.5,2.25)(9.8,2)(10.5,1.75)

\end{pspicture}
\end{center}
\caption{The sunflowers corresponding to matrix entries $-1$ (left), $2$ (center), and $-2$ (right).}
\label{sunflowers}
\end{figure}

{\bf Step 2.} Next we construct for each column $\ab_j$ a connected matched-petal partitioned-core sunflower $\MH_j$. For this we use the sunflowers constructed in {\bf Step 1} in the following way: if $a_{ij}>0$ we use the previously constructed sunflower, otherwise the sunflower obtained by deleting the petal with two vertices from the corresponding sunflower. Finally, $\MH_j$ consists of these disjoint 3-uniform sunflowers, and a perfect matching on the union of non-core vertices. Then we add to $\MH_j$ the deleted petals with two vertices to obtain $\MH'_j$. For example, since ${\bf a}_1=(-1,-2)$ we take the sunflowers corresponding to $-1$ and $-2$ from {\bf Step 1} (see the leftmost picture of Figure~\ref{matched-petal}), then delete the petals with two vertices to construct the connected matched-petal partition-core sunflower $\MH_1$ (see the middle picture of Figure~\ref{matched-petal}) and finally, add the two petals of cardinality 2 to obtain $\MH'_1$ (the rightmost picture of Figure~\ref{matched-petal}).  

\begin{figure}[hbt]
\label{matched-petal}
\begin{center}
\psset{unit=0.9cm}
\begin{pspicture}(0,0.75)(14,5.5)
 
 \rput(1,5){$\bullet$}
 \rput(2,5){$\bullet$}
 \rput(3,5){$\bullet$}
 \rput(4,5){$\bullet$}
  \rput(2.5,3.75){$+$}
  
 \rput(1,4.5){\tiny $v_2(11)$}
 \rput(2,4.5){\tiny $v_1(11)$}
 \rput(3,4.5){\tiny $u(11)$}
 \rput(4,4.5){\small $v_1$}  
 
\psecurve[linewidth=1.1pt, linecolor=black](4.2,5)(3.5,5.25)(2.8,5)(3.5,4.75)(4.2,5)(3.5,5.25)(2.8,5)(3.5,4.75)
\psecurve[linewidth=1.1pt, linecolor=black](3.2,5)(2,5.3)(0.8,5)(2,4.7)(3.2,5)(2,5.3)(0.8,5)(2,4.7)

 \rput(1,1){$\bullet$}
 \rput(1,3){$\bullet$}
 \rput(2,1.5){$\bullet$}
 \rput(2,2.5){$\bullet$}
 \rput(3,2){$\bullet$}
 \rput(4,2){$\bullet$}
 
 \rput(1,0.5){\tiny $v_4(21)$}
 \rput(1,3.5){\tiny $v_2(21)$}
 \rput(2.5,1){\tiny $v_3(21)$}
 \rput(2.5,3){\tiny $v_1(21)$}
 \rput(3.4,1.5){\tiny $u(21)$}
 \rput(4.2,1.5){\small $v_2$}

\psecurve[linewidth=1.1pt, linecolor=black](3.2,1.9)(2.15,2.8)(0.8,3.1)(1.85,2.2)(3.2,1.9)(2.15,2.8)(0.8,3.1)(1.85,2.2)
\psecurve[linewidth=1.1pt, linecolor=black](3.2,2.1)(1.85,1.8)(0.8,0.9)(2.15,1.2)(3.2,2.1)(1.85,1.8)(0.8,0.9)(2.15,1.2)
\psecurve[linewidth=1.1pt, linecolor=black](4.2,2)(3.5,2.25)(2.8,2)(3.5,1.75)(4.2,2)(3.5,2.25)(2.8,2)(3.5,1.75)

\rput(5,3){$\longrightarrow$}

 \rput(7,4){$\bullet$}
 \rput(8,4){$\bullet$}
 \rput(9,4){$\bullet$}

 \rput(7,4.5){\tiny $v_2(11)$}
 \rput(8,4.5){\tiny $v_1(11)$}
 \rput(9,4.5){\tiny $u(11)$}
 
\psecurve[linewidth=1.1pt, linecolor=black](9.2,4)(8,4.3)(6.8,4)(8,3.7)(9.2,4)(8,4.3)(6.8,4)(8,3.7)

 \rput(7,1){$\bullet$}
 \rput(7,3){$\bullet$}
 \rput(8,1.5){$\bullet$}
 \rput(8,2.5){$\bullet$}
 \rput(9,2){$\bullet$}
 
 \rput(7,0.5){\tiny $v_4(21)$}
 \rput(6.5,3.4){\tiny $v_2(21)$}
 \rput(8.5,1){\tiny $v_3(21)$}
 \rput(8.6,3){\tiny $v_1(21)$}
 \rput(9.7,2){\tiny $u(21)$}

\psecurve[linewidth=1.1pt, linecolor=black](9.2,1.9)(8.15,2.8)(6.8,3.1)(7.85,2.2)(9.2,1.9)(8.15,2.8)(6.8,3.1)(7.85,2.2)
\psecurve[linewidth=1.1pt, linecolor=black](9.2,2.1)(7.85,1.8)(6.8,0.9)(8.15,1.2)(9.2,2.1)(7.85,1.8)(6.8,0.9)(8.15,1.2)

\psecurve[linewidth=1.1pt, linecolor=blue](8.15,1.3)(7.7,2.35)(6.85,3.2)(7.3,2.15)(8.15,1.3)(7.7,2.35)(6.85,3.2)(7.3,2.15)
\psecurve[linewidth=1.1pt, linecolor=blue](8.15,2.3)(7.7,3.35)(6.85,4.2)(7.3,3.15)(8.15,2.3)(7.7,3.35)(6.85,4.2)(7.3,3.15)
\psecurve[linewidth=1.1pt, linecolor=blue](8.15,4.2)(7.35,2.65)(6.85,0.8)(7.65,2.35)(8.15,4.2)(7.35,2.65)(6.85,0.8)(7.65,2.35)

\rput(10,3){$\longrightarrow$}

 \rput(11,4){$\bullet$}
 \rput(12,4){$\bullet$}
 \rput(13,4){$\bullet$}
 \rput(14,4){$\bullet$}
  
 \rput(11,4.5){\tiny $v_2(11)$}
 \rput(12,4.5){\tiny $v_1(11)$}
 \rput(13,4.5){\tiny $u(11)$}  
 \rput(14,4.5){\small $v_1$}

\psecurve[linewidth=1.1pt, linecolor=black](14.2,4)(13.5,4.25)(12.8,4)(13.5,3.75)(14.2,4)(13.5,4.25)(12.8,4)(13.5,3.75)
\psecurve[linewidth=1.1pt, linecolor=black](13.2,4)(12,4.3)(10.8,4)(12,3.7)(13.2,4)(12,4.3)(10.8,4)(12,3.7)

 \rput(11,1){$\bullet$}
 \rput(11,3){$\bullet$}
 \rput(12,1.5){$\bullet$}
 \rput(12,2.5){$\bullet$}
 \rput(13,2){$\bullet$}
 \rput(14,2){$\bullet$}
 
 \rput(11,0.5){\tiny $v_4(21)$}
 \rput(10.5,3.4){\tiny $v_2(21)$}
 \rput(12.5,1){\tiny $v_3(21)$}
 \rput(12.6,3){\tiny $v_1(21)$}
 \rput(13.4,1.5){\tiny $u(21)$} 
 \rput(14.2,1.5){\small $v_2$}

\psecurve[linewidth=1.1pt, linecolor=black](13.2,1.9)(12.15,2.8)(10.8,3.1)(11.85,2.2)(13.2,1.9)(12.15,2.8)(10.8,3.1)(11.85,2.2)
\psecurve[linewidth=1.1pt, linecolor=black](13.2,2.1)(11.85,1.8)(10.8,0.9)(12.15,1.2)(13.2,2.1)(11.85,1.8)(10.8,0.9)(12.15,1.2)
\psecurve[linewidth=1.1pt, linecolor=black](14.2,2)(13.5,2.25)(12.8,2)(13.5,1.75)(14.2,2)(13.5,2.25)(12.8,2)(13.5,1.75)

\psecurve[linewidth=1.1pt, linecolor=blue](12.15,1.3)(11.7,2.35)(10.85,3.2)(11.3,2.15)(12.15,1.3)(11.7,2.35)(10.85,3.2)(11.3,2.15)
\psecurve[linewidth=1.1pt, linecolor=blue](12.15,2.3)(11.7,3.35)(10.85,4.2)(11.3,3.15)(12.15,2.3)(11.7,3.35)(10.85,4.2)(11.3,3.15)
\psecurve[linewidth=1.1pt, linecolor=blue](12.15,4.2)(11.35,2.65)(10.85,0.8)(11.65,2.35)(12.15,4.2)(11.35,2.65)(10.85,0.8)(11.65,2.35)

\end{pspicture}
\end{center}
\caption{The bouquet with basis $\MH'_1$}
\label{matched-petal}
\end{figure}
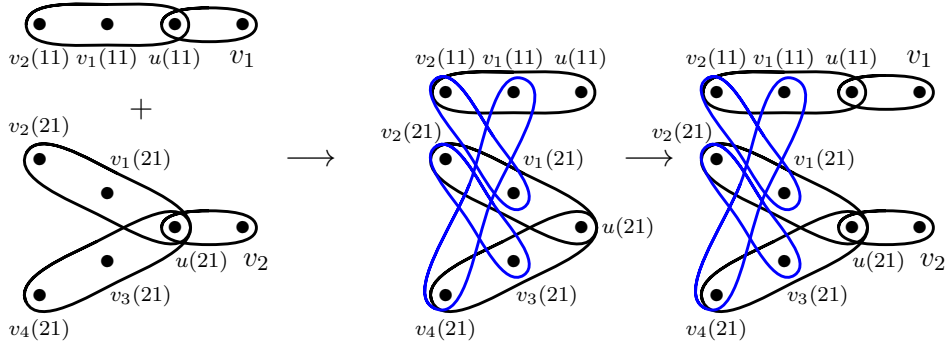

For sake of completeness, we have labeled the vertices of sunflowers from Figure~\ref{matched-petal} according to the general construction of Theorem~\ref{worse}, and we denote their edges by $e_0(11)=\{u(11),v_1\}$, $e_1(11)=\{u(11),v_1(11),v_2(11)\}$, $e_0(21)=\{u(21),v_2\}$, $e_1(21)=\{u(21),v_1(21),v_2(21)\}$, $e_2(21)=\{u(21),v_3(21),v_4(21)\}$, $e_1=\{v_2(11),v_1(21)\}$, $e_2=\{v_2(21),v_3(21)\}$, $e_3=\{v_4(21),v_1(11)\}$. One should note that there is not a unique way to choose the blue edges (in the picture $e_2,e_3,e_4$) such that they form a matching of the non-core vertices of the connected matched-petal partitioned-core sunflower $\MH_1$. Also, we see that the hypergraph $\MH'_1$ constructed is a bouquet with basis the set $U_1$ of all vertices except $v_1,v_2$. Similarly one repeats the construction for each one of the column vectors of $A$, that is ${\bf a}_2,{\bf a}_3,{\bf a}_4$, obtaining the hypergraphs depicted in Figure~\ref{all_matched_petal}. Note that all four hypergraphs $\MH'_1,\ldots,\MH'_4$ obtained through the previous construction are bouquets with basis $U_1,\ldots,U_4$, where each $U_i$ is just the set of vertices of $\MH'_i$ except $v_1,v_2$.  

\begin{figure}
\begin{center}
\psset{unit=0.9cm}
\begin{pspicture}(0,0.75)(13,4.5)

 \rput(1,4){$\bullet$}
 \rput(2,4){$\bullet$}
 \rput(3,4){$\bullet$}
 \rput(4,4){$\bullet$}
 \rput(1,4.5){\tiny $v_2(12)$}
 \rput(2,4.5){\tiny $v_1(12)$}
 \rput(3,4.5){\tiny $u(12)$}
 \rput(4,4.5){\small $v_1$}

\psecurve[linewidth=1.1pt, linecolor=black](4.2,4)(3.5,4.25)(2.8,4)(3.5,3.75)(4.2,4)(3.5,4.25)(2.8,4)(3.5,3.75)
\psecurve[linewidth=1.1pt, linecolor=black](3.2,4)(2,4.3)(0.8,4)(2,3.7)(3.2,4)(2,4.3)(0.8,4)(2,3.7)

 \rput(1,1){$\bullet$}
 \rput(1,3){$\bullet$}
 \rput(2,1.5){$\bullet$}
 \rput(2,2.5){$\bullet$}
 \rput(3,2){$\bullet$}
 
  \rput(1,0.5){\tiny $v_4(22)$}
 \rput(0.5,3.4){\tiny $v_2(22)$}
 \rput(2.5,1){\tiny $v_3(22)$}
 \rput(2.5,3){\tiny $v_1(22)$}
 \rput(3.5,2){\small $v_2$}

\psecurve[linewidth=1.1pt, linecolor=black](3.2,1.9)(2.15,2.8)(0.8,3.1)(1.85,2.2)(3.2,1.9)(2.15,2.8)(0.8,3.1)(1.85,2.2)
\psecurve[linewidth=1.1pt, linecolor=black](3.2,2.1)(1.85,1.8)(0.8,0.9)(2.15,1.2)(3.2,2.1)(1.85,1.8)(0.8,0.9)(2.15,1.2)

\psecurve[linewidth=1.1pt, linecolor=blue](2.15,1.3)(1.7,2.35)(0.85,3.2)(1.3,2.15)(2.15,1.3)(1.7,2.35)(0.85,3.2)(1.3,2.15)
\psecurve[linewidth=1.1pt, linecolor=blue](2.15,2.3)(1.7,3.35)(0.85,4.2)(1.3,3.15)(2.15,2.3)(1.7,3.35)(0.85,4.2)(1.3,3.15)
\psecurve[linewidth=1.1pt, linecolor=blue](2.15,4.2)(1.35,2.65)(0.85,0.8)(1.65,2.35)(2.15,4.2)(1.35,2.65)(0.85,0.8)(1.65,2.35)

 \rput(6,1){$\bullet$}
 \rput(7,1){$\bullet$}
 \rput(8,1){$\bullet$}
 \rput(9,1){$\bullet$}
 \rput(6,0.5){\tiny $v_2(23)$}
 \rput(7,0.5){\tiny $v_1(23)$}
 \rput(8,0.5){\tiny $u(23)$}
 \rput(9,0.5){\small $v_2$}

\psecurve[linewidth=1.1pt, linecolor=black](9.2,1)(8.5,1.25)(7.8,1)(8.5,0.75)(9.2,1)(8.5,1.25)(7.8,1)(8.5,0.75)
\psecurve[linewidth=1.1pt, linecolor=black](8.2,1)(7,1.3)(5.8,1)(7,0.7)(8.2,1)(7,1.3)(5.8,1)(7,0.7)

 \rput(6,2){$\bullet$}
 \rput(6,4){$\bullet$}
 \rput(7,2.5){$\bullet$}
 \rput(7,3.5){$\bullet$}
 \rput(8,3){$\bullet$}
 
 \rput(5.4,1.7){\tiny $v_4(13)$}
 \rput(6,4.5){\tiny $v_2(13)$}
 \rput(7.5,2){\tiny $v_3(13)$}
 \rput(7.5,4){\tiny $v_1(13)$}
 \rput(8.5,3){\small $v_1$} 
 
\psecurve[linewidth=1.1pt, linecolor=black](8.2,2.9)(7.15,3.8)(5.8,4.1)(6.85,3.2)(8.2,2.9)(7.15,3.8)(5.8,4.1)(6.85,3.2)
\psecurve[linewidth=1.1pt, linecolor=black](8.2,3.1)(6.85,2.8)(5.8,1.9)(7.15,2.2)(8.2,3.1)(6.85,2.8)(5.8,1.9)(7.15,2.2)

\psecurve[linewidth=1.1pt, linecolor=blue](7.15,2.3)(6.7,3.35)(5.85,4.2)(6.3,3.15)(7.15,2.3)(6.7,3.35)(5.85,4.2)(6.3,3.15)
\psecurve[linewidth=1.1pt, linecolor=blue](7.15,0.8)(6.7,1.65)(5.85,2.2)(6.3,1.35)(7.15,0.8)(6.7,1.65)(5.85,2.2)(6.3,1.35)
\psecurve[linewidth=1.1pt, linecolor=blue](7.15,3.7)(6.3,2.4)(5.85,0.8)(6.7,2.10)(7.15,3.7)(6.3,2.4)(5.85,0.8)(6.7,2.10)

 \rput(11,1.5){$\bullet$}
 \rput(11,3.5){$\bullet$}
 \rput(12,2){$\bullet$}
 \rput(12,3){$\bullet$}
 \rput(13,2.5){$\bullet$}
 
 \rput(11,1){\tiny $v_4(14)$}
 \rput(11,4){\tiny $v_2(14)$}
 \rput(12.5,1.5){\tiny $v_3(14)$}
 \rput(12.5,3.5){\tiny $v_1(14)$}
 \rput(13.5,2.5){\small $v_1$}

\psecurve[linewidth=1.1pt, linecolor=black](13.2,2.4)(12.15,3.3)(10.8,3.6)(11.85,2.7)(13.2,2.4)(12.15,3.3)(10.8,3.6)(11.85,2.7)
\psecurve[linewidth=1.1pt, linecolor=black](13.2,2.6)(11.85,2.3)(10.8,1.4)(12.15,1.7)(13.2,2.6)(11.85,2.3)(10.8,1.4)(12.15,1.7)
\psecurve[linewidth=1.1pt, linecolor=blue](11,1.3)(11.35,2.5)(11,3.7)(10.65,2.5)(11,1.3)(11.35,2.5)(11,3.7)(10.65,2.5)
\psecurve[linewidth=1.1pt, linecolor=blue](12,1.8)(12.2,2.5)(12,3.2)(11.8,2.5)(12,1.8)(12.2,2.5)(12,3.2)(11.8,2.5)

\end{pspicture}
\end{center}
\caption{The other 3 bouquets with basis corresponding to ${\bf a}_2,{\bf a}_3,{\bf a}_4$.}
\label{all_matched_petal}
\end{figure}
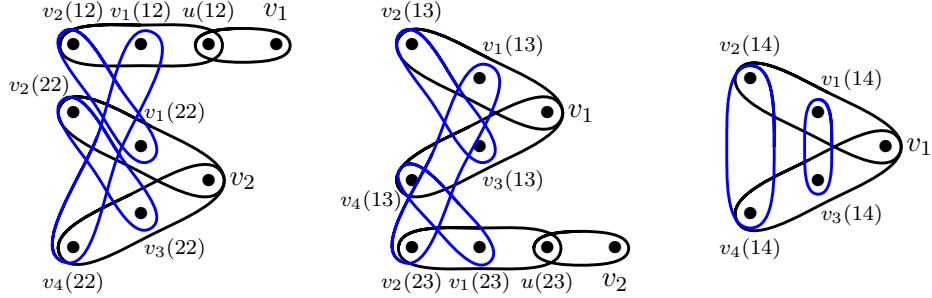

{\bf Step 3} Finally, consider the hypergraph $\MH=(V,\ME)$ with $|V|=28$ and $|\ME|=26$, obtained from $\MH'_1,\ldots,\MH'_4$ by taking the vertex set to be the union of the set of vertices of $\MH'_1,\ldots,\MH'_4$, and the set of edges is the (disjoint) union of the set of edges of $\MH'_1,\ldots,\MH'_4$. Furthermore, we compute the subbouquet vectors of $\MH'_1,\ldots,\MH'_4$. For this, we analyze in detail only the first hypergraph $\MH'_1$ which is a bouquet with basis $U_1$, constructed in Figure~\ref{matched-petal}, the others being computed similarly. We assume that the column vectors of the incidence matrix of $\MH$ are indexed such that the first eight correspond, in this order, to the edges $e_1(11),e_0(11),e_1(21),e_2(21),e_0(21),e_1,e_2,e_3$, while the first two rows are indexed by the vertices $v_1$ and $v_2$. Then $\cb_{\MH'_1}=(1, -1, 1, 1, -2, -1, -1, -1, {\bf 0})\in\ZZ^{26}$ and $$\ab_{\MH'_1}=e_1(11)-e_0(11)+e_1(21)+e_2(21)-2e_0(21)-e_1-e_2-e_3=(-1,-2,{\bf 0})\in\ZZ^{28}.$$

Analogously, with respect to a similar order for the rest of the edges of the other three bouquets, we obtain the following
${\bf c}_{\MH'_2}=({\bf 0},1,-1,1,1,-1,-1,-1,{\bf 0},{\bf 0})$, ${\bf c}_{\MH'_3}=({\bf 0},{\bf 0},1,1,1,-1,-1,-1,-1,{\bf 0})$ and ${\bf c}_{\MH'_4}=({\bf 0},{\bf 0}, {\bf 0}, 1,1,-1,-1)\in\ZZ^{26}$. Therefore the other three bouquet vectors are ${\bf a}_{\MH'_2}=(-1,2,{\bf 0})$, ${\bf a}_{\MH'_3}=(2,-1,{\bf 0})$ and ${\bf a}_{\MH'_4}=(2,0,{\bf 0})\in\ZZ^{28}$. Since the set of edges of $\MH$ can be partitioned into four bouquets with bases $\MH_1',\ldots,\MH_4'$, the ideal $I_{\MH}$ is strongly robust, by Corollary~\ref{no_free}. By Theorem~\ref{all_is_well}, the bijective correspondence between the Graver basis of $I_A$ and the Graver basis of $I_{\MH}$ applies in the following way: to the vector ${\bf u}=(1,3,4,-2)\in\Gr(A)$ corresponds the vector $B({\bf u})$
\[
(1, -1, 1, 1, -2, -1, -1, -1| 3,-3,3,3,-3,-3,-3| 4,4,4,-4,-4,-4,-4| 2,2,-2,-2). 
\]    

There are seven elements in $\Gr({\MH})$.}
\end{Example}

 For the next construction we may assume that the matrix $A$ has no zero column or row.  Indeed, if, say, $\ab_1={\bf 0}$, then $I_A=(I_{A'},1-x_1)$, where $I_{A'}\subset K[x_2,\ldots,x_n]$, and if, say, the first row is zero, then $I_A=I_{A'}$, where $A'=[\ab'_1,\ldots,\ab'_n]$, with $\ab'_j=(a_{2j},\ldots,a_{mj})\subset\ZZ^{m-1}$ for all $j$.    

\begin{Theorem}\label{worse}
Given any integer matrix $A$ without any zero row or zero column,  there exists an almost 3-uniform hypergraph $\MH=(V,\ME)$ such that: 
\begin{enumerate}
\item[(1)] There is a bijective correspondence, ${\bf u}\mapsto B({\bf u})$, between $\Ker_{\ZZ}(A)$, $\Gr(A)$ and $\MC(A)$ and $\Ker_{\ZZ}(\MH)$, $\Gr(\MH)$ and $\MC(A)$, respectively. 
\item[(2)] For every ${\bf u}\in\Gr(A)$ we have $|\ub|_1\leq |B(\ub)|_1$, where $|\ub|_1=\sum_{i=1}^n |u_i|$ represents the $1$-norm of the vector $\ub$. 
\item[(3)] The toric ideal $I_{\MH}$ is strongly robust.
\end{enumerate}
\end{Theorem}
\begin{proof}
Let $A=[{\bf a}_1,\ldots,{\bf a}_n]\in\ZZ^{m\times n}$. We will  construct  a hypergraph $\MH$ whose \emph{subbouquet ideal} is $I_A$, and with the property that all its non-free subbouquets are mixed. This will imply at once conditions (1) and (3), by Theorem~\ref{all_is_well} and   Corollary~\ref{no_free}. To this end, let $\{v_1, \dots, v_m\}$ be a set of vertices, and for each nonzero entry of the matrix $a_{ij}$ we introduce the following new vertices: 
\begin{enumerate}
 \item[(1)] if $a_{ij}>0$ the set $\{v_1(ij),v_2(ij), \cdots, v_{2a_{ij}}(ij)\}$ of $2a_{ij}$ vertices,
 \item[(2)] if $a_{ij}<0$ the set $\{u(ij), v_1(ij),v_2(ij), \cdots, v_{-2a_{ij}}(ij)\}$ of $-2a_{ij}+1$ vertices,
 \end{enumerate}
as well as the following edges: 
\begin{enumerate}
 \item[(3)] if $a_{ij}>0$ the $a_{ij}$ edges $e_s(ij)=\{v_i, v_{2s-1}(ij), v_{2s}(ij)\}$, where $1\leq s\leq a_{ij}$,
 \item[(4)] if $a_{ij}<0$ the $-a_{ij}+1$ edges $e_s(ij)=\{u(ij), v_{2s-1}(ij), v_{2s}(ij)\}$, where $1\leq s\leq -a_{ij}$ and $e_0(ij)=\{v_i, u(ij)\}$.
\end{enumerate}

Note that for each $a_{ij}\neq 0$ the hypergraph $\MS_{ij}$, on the set of vertices  defined in item (1) (item (2), respectively) with the set of edges defined in item (3) (item (4), respectively) is a sunflower with the core $v_i$ ($u(ij)$, respectively). Furthermore, each sunflower $\MS_{ij}$ is either a 3-uniform sunflower if $a_{ij}>0$ or an almost 3-uniform sunflower if $a_{ij}<0$ (by almost we mean that all the edges have three vertices except $e_0(ij)$, which has only two vertices). In addition, for any fixed $j$ the sunflowers $\MS_{1j},\ldots,\MS_{mj}$ are vertex disjoint by definition. For each nonzero column $\ab_j=(a_{1j},\ldots,a_{mj})$ we construct first a connected matched-petal partitioned-core sunflower $\MH_j=(V_j,\ME_j)$ on the set of vertices
\[
V_j=\bigcup_{i:a_{ij}>0} \{v_1(ij),v_2(ij),\ldots, v_{2a_{ij}}(ij), v_i\}\cup \bigcup_{i:a_{ij}<0}\{v_1(ij),v_2(ij),\ldots, v_{-2a_{ij}}(ij), u(ij)\}, 
\]
with the core vertices 
\[
C_j=\bigcup_{i:a_{ij}>0}\{v_i\}\cup\bigcup_{i:a_{ij}<0}\{u(ij)\}.
\]

 For this is enough to describe the construction of $\MH_1=(V_1,\ME_1)$, the others being similar. Consider all nonzero entries of the column $\ab_1=(a_{11},\ldots,a_{m1})$. To each nonzero $a_{i1}$ we associate the sunflower $\MS_{i1}$ if $a_{i1}>0$, or the sunflower denoted by $\MS_{i1}\setminus\{e_0(i1)\}$ obtained from $\MS_{i1}$ by removing the edge $e_0(i1)$, if $a_{i1}<0$. Note that $\MS_{i1}\setminus\{e_0(i1)\}$ is a 3-uniform sunflower with core the set $\{u(i1)\}$, and its vertex set is obtained from the vertex set of $\MS_{i1}$ by removing $v_i$. The matched-petal partitioned-core sunflower $\MH_1$ with vertex set $V_1$, consists of the vertex-disjoint sunflowers $\MS_{i1}$ (for those $i$ with $a_{i1}>0$) and $\MS_{i1}\setminus\{e_0(i1)\}$  (for those $i$ with $a_{i1}<0$), and with the following perfect matching on the set $V_j\setminus C_j$, of non-core vertices:
\[
\{v_2(11),v_3(11)\}, \ldots, \{v_{2|a_{11}|-2}(11),v_{2|a_{11}|-1}(11)\}, \{v_{2|a_{11}|}(11),v_1(21)\},
\]
\[
 \{v_2(21),v_3(21)\}, \ldots, \{v_{2|a_{m1}|-2}(m1),v_{2|a_{m1}|-1}(m1)\}, \{v_{2|a_{m1}|}(m1),v_1(11)\},         
\]
where for convenience of notation we assumed that all $a_{i1}\neq 0$ (see for an example the blue edges from Figures~\ref{matched-petal},~\ref{all_matched_petal}). This perfect matching on the non-core vertices ensures that $\MH_j=(V_j,\ME_j)$ is a connected matched-petal partitioned-core sunflower.  

Then we ``extend" the hypergraph $\MH_j$ to the hypergraph $\MH'_j=(V'_j,\ME'_j)$, where $V'_j=V_j\cup \{v_i: a_{ij}<0\}$ and $\ME'_j=\ME_j\cup \{e_0(ij): a_{ij}<0\}$, which turns out to be a bouquet with basis the set $U_j=V'_j\setminus\{v_i: a_{ij}\neq 0\}$. Alternatively, we can write $\MH'_j=\ME_{U_j}$ for all $j=1,\ldots,n$. Finally, the hypergraph $\MH=(V,\ME)$ we are looking for is obtained from $\MH'_1,\ldots,\MH'_n$ as follows $V=\bigcup_{j=1}^n V'_j$, and $\ME=\bigcup_{j=1}^n \ME'_j$. Note that since $A$ has no zero row, then $\{v_1,\ldots,v_m\}\subset V$. Since each $\MH'_j$ is a bouquet with basis, then it follows immediately from construction that $\MH$ has $n$ bouquets with bases. Moreover, the resulting bouquet vector of $\MH'_j=\ME_{U_j}$ is $\ab_{\ME_{U_j}}=(\ab_j,{\bf 0})\in\ZZ^{|V|}$ and thus we obtain that the subbouquet ideal of $I_{\MH}$ is $I_A$, as desired.  \qed

\end{proof}


 The following two results were inspired by the universality results of \cite{DO}.  In some sense, the following two corollaries strengthen \cite[Theorem 1.2]{DO} and \cite[Corollaries 2.1,2.2]{DO}. There, the motivation from algebraic statistics was to  show that, for a particular $0/1$ matrix of a $3$-uniform hypergraph encoding a model on  three-dimensional contingency tables,   as two of the dimensions of the  table grow, there cannot be a universal upper bound on the degrees of minimal Markov bases. To show this, it is enough for \cite{DO} to construct one element of given arbitrarily high degree. 
The results below differ in three ways: 1) the underlying incidence matrices of the almost $3$-uniform hypergraphs are $0-1$ matrices with at most three ones on each column; 2) our statement holds true for all binomials in any minimal Markov basis, and not only for one binomial; 3) since the toric ideal constructed is strongly robust then the statement is also valid for Graver basis, indispensable binomials, universal Gr\" obner basis respectively. 
The reader should note that  the two classes of matrices considered in \cite{DO} and considered here have in common only the fact that they are $0-1$ matrices with at most three ones on each column.

\begin{Corollary}\label{universality_markov}
For any vectors $d_1,\ldots,d_r \in\ZZ^m$ there exists an almost 3-uniform hypergraph $\MH$ with its toric ideal strongly robust such that all of the elements of the minimal Markov basis of $\MH$ restricted to some of its entries cover the whole set $\{d_1,\ldots,d_r\}$. 
\end{Corollary}
\begin{proof}
Let $D=[{\pmb\alpha}_1,\ldots,{\pmb\alpha}_r]\in\ZZ^{p\times r}$ be a unimodular matrix such that all of its bouquets are not free, and consider the vectors $\cb'_i=(1,{\bf d}_i)=(1,d_{i1},\ldots,d_{im})\in\ZZ^{m+1}$ for all $i=1,\ldots,r$. By definition the vectors $\cb'_1,\ldots,\cb'_r$ are primitive and thus we can apply Theorem~\ref{inverse_construction} in order to obtain a generalized Lawrence matrix $A=[\ab_1,\ldots,\ab_n]\in\ZZ^{q\times n}$ such that its subbouquet ideal equals $I_D$, and whose Graver basis equals the set of circuits and implicitly the universal Gr\" obner basis. In addition, we know that the subbouquet $B_i$ is encoded by the vectors $\ab_{B_i}=({\pmb\alpha}_i,{\bf 0},\ldots,{\bf 0})$ and $\cb_{B_i}=({\bf 0},\ldots,\cb'_i,\ldots,{\bf 0})$, for all $i=1,\ldots,r$. Since $D$ is unimodular applying \cite[Proposition 8.11]{St} we have that the Graver basis elements of $D$ are vectors with coordinates only $0,1,-1$, which in turn implies via the one-to-one correspondence of Theorem~\ref{all_is_well} 
\[
B((u'_1,\ldots,u'_r)) = \sum_{i=1}^r \cb_{B_i} u'_i = (u'_1,u'_1{\bf d}_1,\ldots,u'_r,u'_r{\bf d}_r)\in\ZZ^n
\]
that the Graver basis elements of $A$ are vectors whose coordinate nonzero blocks are either $\cb'_i$ or $-\cb'_i$.  

Finally, we apply for the matrix $A$ the construction of Theorem~\ref{worse} to obtain the hypergraph $\MH$ which has $n$ bouquets with bases $\MH'_1,\ldots,\MH'_n$, and whose toric ideal is strongly robust. Note from the proof of Theorem~\ref{worse} that the nonzero coordinates of the vectors $\cb_{\MH'_i}$ corresponding to the edges of the perfect matching on the non-core vertices are equal either to $1$ or $-1$.  Since the subbouquet ideal of $I_{\MH}$ is equal to $I_A$ it follows from the description of the Graver basis elements of $A$ that via the one-to-one correspondence 
\[
(u'_1,u'_1d_{11},\ldots,u'_1d_{1m},\ldots,u'_r,u'_rd_{r1},\ldots,u'_rd_{rm}) = (u_1,\ldots,u_n)\mapsto \sum_{j=1}^n \cb_{\MH'_j}u_j
\]    
the Graver basis elements of $\MH$ contain as subvectors at least one of $\pm d_1,\ldots,\pm d_r$. Indeed, from the construction of $A$ we have that $u'_i\in\{-1,0,1\}$ and we just substitute them in the previous displayed formula. Since $I_{\MH}$ is strongly robust then the Graver basis equals the minimal Markov basis and thus we have the desired conclusion. 

For the final part of statement, that is all of the vectors $d_1,\ldots,d_r$ appear as a support of some elements from the minimal Markov basis of $\MH$ we need to make a choice on the unimodular matrix $D$ to ensure that in the Graver basis of $I_D$ appear all of the variables. For example, we may choose $D$ to be the incidence matrix of a complete bipartite graph $K_{p,\ell}$ with $p,\ell\geq 2$ if $r$ equals its number of edges. Otherwise, we consider a complete bipartite graph $K_{p,\ell}$ with the number of edges $q>r$, let $D\in\ZZ^{p\times q}$ be its incidence matrix and consider the $q$ vectors $\cb'_1,\ldots,\cb'_r,\cb'_r,\ldots,\cb'_r$ with the vector $\cb'_r$ being repeated $q-r+1$ times. \qed

\end{proof}

Taking $d_1=\cdots=d_r=d$ in the previous corollary and switching from the vector notation of an element $\vb\in\Ker_{\ZZ}(\MH)$ to the binomial notation $x^{\vb^+}-x^{\vb^-}\in I_{\MH}$ we obtain the following

\begin{Corollary}
\label{universality}
For any vector $d\in\ZZ^n$ there exists an almost 3-uniform hypergraph $\MH$ with its toric ideal strongly robust such that all binomials in the minimal generating set of $I_H$ satisfy the following: one of its monomials restricted to a suitable subset of variables has multi-degree $d^+$ and its other monomial restricted to a suitable subset of variables has multi-degree $d^-$. 
\end{Corollary}

\section{ Hypergraphs encode all positively graded toric ideals }
\label{sec:complexity2}


This section presents a correspondence between positively graded (general) toric ideals and stable toric ideals of hypergraphs. Namely, given a positively graded $I_A$, Theorem~\ref{general_complexity} constructs a stable toric ideal of a hypergraph $\MH=\MH(A)$ that has the same combinatorial complexity and whose homological properties are preserved. In particular, the Graver bases of the ideals $I_A$ and $I_\MH$ have the same number of elements, and the same holds for their universal Gr\"obner and Markov bases, as well as indispensable binomials and circuits. Even more is true: if $I_A$ has $\ell$ different minimal Markov bases then $I_\MH$ has also $\ell$ different minimal Markov bases, and the same holds for the reduced Gr\"{o}bner bases. Contrast this with Section~\ref{sec:complexity}, where an arbitrary toric ideal may have infinitely many minimal Markov bases.

For the remainder of this section, fix the following notation: 
\[
	\Sigma_n:=\left( \begin{matrix}
	0&1&\dots&\dots&1\\1&0&1&\dots&1 \\ \vdots &\ddots&\ddots&\ddots&\vdots\\  1&\dots&1&0&1 \\ 1&\dots&\dots&1&0
	\end{matrix}\right)\in\ZZ^{n\times n}
\]
and  
\[ \varepsilon_{k,n}:=(\underbrace{1 \ \cdots 1}_{k \text{ times }} \ 0 \ \cdots \ 0)\in\ZZ^n, \text{ where } 1\leq k<n.\]

The following example illustrates a general construction which is the basis for  the proof of Theorem~\ref{general_complexity}.

\begin{Example}
{\em Given a non-negative integer matrix $D$, we seek a procedure to create a $0/1$ matrix $A$ such that there is a bijective correspondence between distinguished sets of binomials of $I_D$ and $I_A$, namely, each of: Graver basis, circuits, indispensable binomials, minimal Markov bases, reduced Gr\"obner bases (universal Gr\"obner basis).    

Let 
\[
D=(d_{ij})=\left( \begin{array}{ccccc}
1 & 3 & 2 & 0 & 1\\
3 & 2 & 1 & 3 & 2\\
3 & 0 & 2 & 2 & 1
\end{array} \right)\in\NN^{3\times 5}.
\]

Let $\delta_i:=\max_l \{d_{li}\}$ be the maximum entry in column $i$ of $D$, and $j_i:=\min\{l : d_{li}=\delta_i\}$ be the index of the first row where $\delta_i$ appears. For the given matrix $D$, these values are as follows:  $\delta_1=3$, $\delta_2=3$, $\delta_3=2$, $\delta_4=3$ and $\delta_5=2$; and $j_1=2$, $j_2=1$, $j_3=1$, $j_4=2$ and $j_5=2$. Define $\delta:=\sum_{i=1}^5 (\delta_i+1)=18$  and $l=3 - |\{j_1,\dots,j_5\}|=1$, the number of rows that do not contain any column-maximum entry $\delta_i$. We construct a 0-1 matrix $A=M_{\MH}$ of size $19\times 18$, the incidence matrix of a hypergraph $\MH$, such that its subbouquet ideal will be $I_D$. Here, $19=\delta+l$ and $18=\delta$. The matrix $A$ is constructed from $D$ in two steps.

{\bf Step 1.} Every column-maximum entry $\delta_i$ defined above will determine a set of horizontal blocks of the matrix $A$ as follows. For each  row index $k\in\{j_i\}_{i=1}^{5}$, consider the set of all $\delta_i$'s appearing on the $k$-th row of $D$. For each such $\delta_i$, construct a $(\delta_i+1)\times\delta$ block matrix by concatenating (horizontally) the following $5$ sub-blocks: block $i$ shall consist of the matrix $\Sigma_{\delta_i+1}$, while for each $l\neq i$, block $l$ shall consist of $\delta_i+1$ copies of $\varepsilon_{d_{kl},\delta_i+1}$.

 For example, the first row of $D$ contains two column-maximum entries, $\delta_2$ and $\delta_3$. The first entry will generate a row-block of size $(3+1)\times 18$ inside $A$, while the second entry will generate a row-block of size $(2+1)\times 18$. First, the entry $\delta_2=3$ requires concatenating five sub-blocks: the first  one consisting of $4=\delta_2+1$ copies of $\varepsilon_{d_{11},4}=(1 \ 0 \ 0 \ 0)$; the second one, $\Sigma_{4}$; the third one consisting of $4$ copies of $\varepsilon_{d_{13},4}=(1 \ 1 \ 0 \ 0)$; the fourth one, $4$ copies of $\varepsilon_{d_{14},4}=(0 \ 0 \ 0 \ 0)$; and the fifth, $4$ copies of  $\varepsilon_{d_{15},4}=(1 \ 0 \ 0 \ 0)$. This resulting row-block of $A$ is displayed below, with the distinguished sub-block corresponding to $\Sigma_{4}$ colored in blue.
 
\[
\left( \begin{array}{cccc|cccc|ccc|cccc|ccc}
 1 &  0 &  0 &  0 &  {\blue 0} & {\blue 1} & {\blue 1} & {\blue 1} &  1 &  1 &  0 &  0 &  0 &  0 &  0 &  1 & 0 & 0 \\
 1 &  0 &  0 &  0 &  {\blue 1} & {\blue 0} & {\blue 1} & {\blue 1} &  1 &  1 &  0 &  0 &  0 &  0 &  0 &  1 & 0 & 0 \\
 1 &  0 &  0 &  0 &  {\blue 1} & {\blue 1} & {\blue 0} & {\blue 1} &  1 &  1 &  0 &  0 &  0 &  0 &  0 &  1 & 0 & 0 \\
 1 &  0 &  0 &  0 &  {\blue 1} & {\blue 1} & {\blue 1} & {\blue 0} &  1 &  1 &  0 &  0 &  0 &  0 &  0 &  1 & 0 & 0 
\end{array} \right).
\]         
Similarly, the second entry, $\delta_3=2$, will generate the following $3\times 18$ row-block: 
 
\[
\left( \begin{array}{cccc|cccc|ccc|cccc|ccc}
 1 &  0 &  0 &  0 &  1 &  1 &  1 &  0 & {\blue 0} & {\blue 1} & {\blue 1} &  0 &  0 &  0 &  0 &  1 & 0 & 0 \\
 1 &  0 &  0 &  0 &  1 &  1 &  1 &  0 & {\blue 1} & {\blue 0} & {\blue 1} &  0 &  0 &  0 &  0 &  1 & 0 & 0 \\
 1 &  0 &  0 &  0 &  1 &  1 &  1 &  0 & {\blue 1} & {\blue 1} & {\blue 0} &  0 &  0 &  0 &  0 &  1 & 0 & 0
\end{array} \right).
\]

The remaining row indices are $j_1=j_4=j_5=2$, thus we consider the second row of $D$. It contains three $\delta_i$ values: $\delta_1,\delta_4,\delta_5$, and thus gives rise to three row-blocks of similar structure. These are the third, fourth and fifth row-blocks of the matrix $A$ in Equation~\eqref{non_mixed_hypergraph}.

{\bf Step 2.} For each row $k$ of $D$ containing none of the column-maxima $\delta_i$'s, that is, for all row indices $k\in\{1,2,3\}\setminus\{j_i\}_{i=1}^{5}$, we will create one additional row of $A$. This row will be obtained by concatenating (in this order) the following vectors $\varepsilon_{d_{k1},\delta_1+1}$, $\varepsilon_{d_{k2},\delta_2+1}$, $\varepsilon_{d_{k3},\delta_3+1}$, $\varepsilon_{d_{k4},\delta_4+1}$, $ \varepsilon_{d_{k5},\delta_5+1}$. For example, the third row of $D$ does not contain any $\delta_i$ and thus the row of $A$ obtained from concatenating the vectors $\varepsilon_{d_{31},4}=(1 \ 1 \ 1 \ 0)$, $\varepsilon_{d_{32},4}=(0 \ 0 \ 0 \ 0)$, $\varepsilon_{d_{33},3}=(1 \ 1 \ 0)$, $\varepsilon_{d_{34},4}=(1 \ 1 \ 0 \ 0)$ and $\varepsilon_{d_{35},3}=(1 \ 0 \ 0)$ is 
\[
\left( \begin{array}{cccc|cccc|ccc|cccc|ccc}
1 &  1 &  1 &  0 &  0 &  0 &  0 &  0 &  1 &  1 &  0 &  1 &  1 &  0 &  0 &  1 & 0 & 0
\end{array} \right).
\]

 Finally, the matrix $A=[\ab_1,\ldots,\ab_{18}]$ is obtained by concatenating (vertically) the two row-blocks generated by the first row of $D$, the three  row-blocks generated by the second row of $D$, and the row generated by the third row of $D$; see Equation~\eqref{non_mixed_hypergraph}. Note that the submatrix of $A$ generated by the first row of $D$ has $4+3=7$ rows, the one generated by the second row of $D$ has $4+4+3=11$ rows, and the one generated by the third row of $D$ has $1$ row. The matrix $A$ has $5$ non-mixed bouquets such that $\ab_1,\ab_2,\ab_3,\ab_4$ belong to the first bouquet $B_1$, $\ab_5,\ab_6,\ab_7,\ab_8$ to $B_2$, $\ab_9,\ab_{10},\ab_{11}$ to $B_3$, $\ab_{12},\ab_{13},\ab_{14},\ab_{15}$ to $B_4$, and $\ab_{16},\ab_{17},\ab_{18}$ to $B_5$. Moreover all nonzero components of $\cb_{B_1},\ldots,\cb_{B_5}\in\ZZ^{18}$ are 1, and thus $\ab_{B_1}=\ab_1+\cdots+\ab_4\in\ZZ^{19}$ is the transposed vector of 
\[
\left( \begin{array}{ccccccc|ccccccccccc|c}
1 &  1 &  1 &  1 &  1 &  1 &  1 &  3 &  3 &  3 &  3 &  3 &  3 &  3 &  3 &  3 & 3 & 3 & 3
\end{array} \right),
\]  
where the first block has $7$ coordinates, the second one  $11$ coordinates, and the last one $1$ coordinate. Similarly, $\ab_{B_2}=\ab_5+\cdots+\ab_8$ is the transposed vector of 
\[
\left( \begin{array}{ccccccc|ccccccccccc|c}
3 &  3 &  3 &  3 &  3 &  3 &  3 &  2 &  2 &  2 &  2 &  2 &  2 &  2 &  2 &  2 & 2 & 2 & 0
\end{array} \right),
\]
and so on. Thus $I_{A_B}=I_D$, and since $I_A$ is a stable toric ideal the desired bijective correspondence follows from Theorem~\ref{stable_toric}.
 
\begin{eqnarray}\label{non_mixed_hypergraph}
A=M_{\MH} = \left( \begin{array}{cccc|cccc|ccc|cccc|ccc}
 1 &  0 &  0 &  0 &  {\blue 0} & {\blue 1} & {\blue 1} & {\blue 1} &  1 &  1 &  0 &  0 &  0 &  0 &  0 &  1 & 0 & 0 \\
 1 &  0 &  0 &  0 &  {\blue 1} & {\blue 0} & {\blue 1} & {\blue 1} &  1 &  1 &  0 &  0 &  0 &  0 &  0 &  1 & 0 & 0 \\
 1 &  0 &  0 &  0 &  {\blue 1} & {\blue 1} & {\blue 0} & {\blue 1} &  1 &  1 &  0 &  0 &  0 &  0 &  0 &  1 & 0 & 0 \\
 1 &  0 &  0 &  0 &  {\blue 1} & {\blue 1} & {\blue 1} & {\blue 0} &  1 &  1 &  0 &  0 &  0 &  0 &  0 &  1 & 0 & 0 \\
\hline
 1 &  0 &  0 &  0 &  1 &  1 &  1 &  0 & {\blue 0} & {\blue 1} & {\blue 1} &  0 &  0 &  0 &  0 &  1 & 0 & 0 \\
 1 &  0 &  0 &  0 &  1 &  1 &  1 &  0 & {\blue 1} & {\blue 0} & {\blue 1} &  0 &  0 &  0 &  0 &  1 & 0 & 0 \\
 1 &  0 &  0 &  0 &  1 &  1 &  1 &  0 & {\blue 1} & {\blue 1} & {\blue 0} &  0 &  0 &  0 &  0 &  1 & 0 & 0 \\
\hline
{\blue 0} & {\blue 1} & {\blue 1} & {\blue 1} &  1 &  1 &  0 &  0 &  1 &  0 &  0 &  1 &  1 &  1 &  0 &  1 & 1 & 0 \\
{\blue 1} & {\blue 0} & {\blue 1} & {\blue 1} &  1 &  1 &  0 &  0 &  1 &  0 &  0 &  1 &  1 &  1 &  0 &  1 & 1 & 0 \\
{\blue 1} & {\blue 1} & {\blue 0} & {\blue 1} &  1 &  1 &  0 &  0 &  1 &  0 &  0 &  1 &  1 &  1 &  0 &  1 & 1 & 0 \\
{\blue 1} & {\blue 1} & {\blue 1} & {\blue 0} &  1 &  1 &  0 &  0 &  1 &  0 &  0 &  1 &  1 &  1 &  0 &  1 & 1 & 0 \\
\hline
 1 &  1 &  1 &  0 &  1 &  1 &  0 &  0 &  1 &  0 &  0 & {\blue 0} & {\blue 1} & {\blue 1} & {\blue 1} &  1 & 1 & 0 \\
 1 &  1 &  1 &  0 &  1 &  1 &  0 &  0 &  1 &  0 &  0 & {\blue 1} & {\blue 0} & {\blue 1} & {\blue 1} &  1 & 1 & 0 \\
 1 &  1 &  1 &  0 &  1 &  1 &  0 &  0 &  1 &  0 &  0 & {\blue 1} & {\blue 1} & {\blue 0} & {\blue 1} &  1 & 1 & 0 \\
 1 &  1 &  1 &  0 &  1 &  1 &  0 &  0 &  1 &  0 &  0 & {\blue 1} & {\blue 1} & {\blue 1} & {\blue 0} &  1 & 1 & 0 \\
\hline 
 1 &  1 &  1 &  0 &  1 &  1 &  0 &  0 &  1 &  0 &  0 &  1 &  1 &  1 &  0 & {\blue 0} & {\blue 1} & {\blue 1} \\
 1 &  1 &  1 &  0 &  1 &  1 &  0 &  0 &  1 &  0 &  0 &  1 &  1 &  1 &  0 & {\blue 1} & {\blue 0} & {\blue 1} \\
 1 &  1 &  1 &  0 &  1 &  1 &  0 &  0 &  1 &  0 &  0 &  1 &  1 &  1 &  0 & {\blue 1} & {\blue 1} & {\blue 0} \\
\hline
 1 &  1 &  1 &  0 &  0 &  0 &  0 &  0 &  1 &  1 &  0 &  1 &  1 &  0 &  0 &  1 & 0 & 0
\end{array} \right)
\end{eqnarray}
\normalsize
 }

\end{Example}

\begin{Theorem}\label{general_complexity}
Let $I_D$ be an arbitrary positively graded nonzero toric ideal. Then there exists a hypergraph $\MH$ such that there is a bijective correspondence between the Graver bases, all minimal Markov bases, all reduced Gr\"obner bases, circuits, and indispensable binomials of $I_D$ and $I_{\MH}$. Furthermore, all of the homological data of $I_{\MH}$ is inherited by $I_D$ and viceversa.
\end{Theorem}
\begin{proof}
Since $I_D$ is a positively graded toric ideal, we may assume by \cite[Corollary 7.23]{MS} that $D=(d_{ij})\in\NN^{m\times n}$. Furthermore, every column of $D$ is nonzero, since, otherwise, if the $j$-th column is $\bf 0$, then $x_j-1\in I_D$, a contradiction to $I_D$ being positively graded.  We will construct the incidence matrix $A$ of a hypergraph $\MH$, such that its subbouquet ideal is equal to $I_D$. For this define column-maximum entries $\delta_i:=\max\{d_{ji}: \ j=1,\ldots,m\}$ for all $i=1,\ldots,n$ and set $\delta=n+\sum_{i=1}^n \delta_i$. Note that $\delta_i>0$ for every $i=1,\ldots,m$. Moreover, denote by $j_i:=\min\{k: \ d_{ki}=\delta_i\}$ for all $i=1,\ldots,n$ and denote by $l=m-\#\{j_i: 1\leq i\leq n\}$. 

Then the 0-1 matrix $A\in\NN^{(\delta+l)\times \delta}$ consists of several blocks, concatenated vertically, of the following two types: 1) For each $k\in\{j_1,\dots,j_n\}$, and for each column-maximum entry $\delta_i$ located on the $k$-th row of $D$, construct  $\delta_i+1$ rows of $A$ by concatenating (horizontally) $n$ block matrices. Here, the $i$-th block is $\Sigma_{\delta_i+1}$, while for each $l\neq i$ the $l$-th block consists of $\delta_i+1$ copies of the row sub-vector $\varepsilon_{d_{kl},\delta_l+1}$. 2)  For each row index $k\in[m]\setminus \{j_1,\dots,j_n\}$ of $D$, that is, each row not containing any column-maximum entry $\delta_i$, construct a row of $A$ by by concatenating (in this order) the following vectors $$(\varepsilon_{d_{k1},\delta_1+1}| \varepsilon_{d_{k2},\delta_2+1}| \ldots | \varepsilon_{d_{kn},\delta_n+1}).$$

Assume now that the columns of $A$ are labeled ${\bf a}_1,\ldots,{\bf a}_{\delta}$. We prove that the first $\delta_1+1$ column vectors of $A$ belong to the same subbouquet $B_1$, the next $\delta_2+1$ belong to the same subbouquet $B_2$, and so on, until the last $\delta_n+1$ column vectors belong to the same subbouquet $B_n$. We will prove this only for the first $\delta_1+1$ columns of $A$, the other cases being similar. By the definition of $A$ we have that in the submatrix of $A$  determined by the columns $\ab_1,\ldots,\ab_{\delta_1+1}$ there exist integers $i_1,\ldots,i_1+\delta_1=i_2$ such that the matrix corresponding to these rows and the columns $\ab_1,\ldots,\ab_{\delta_1+1}$ is $\Sigma_{\delta_1+1}$. We consider the vectors $\cb_{i,i+1}\in\ZZ^{\delta+l}$ for every $i=i_1,\ldots,i_1+\delta_1-1$, whose only nonzero coordinates are $1$ on position $i$, and $-1$ on position $i+1$. Then the co-vector $(\cb_{i_1,i_1+1}\cdot \ab_1,\ldots,\cb_{i_1,i_1+1}\cdot \ab_{\delta})=(-1,1,0,\ldots,0)$ has support of cardinality two and $G(\ab_{1})=G(\ab_{2})$. Here, we used the fact that in the horizontal block containing $\Sigma_{\delta_1+1}$ (and corresponding to the rows $i_1,\ldots,i_2$) all the entries corresponding to the rows $i_1,i_1+1$ and columns $3,\ldots,\delta$ are identical, by definition. Similarly, using all the vectors $\cb_{i,i+1}$ for $i=i_1,\ldots,i_2-1$, we obtain $G(\ab_1)=\cdots=G(\ab_{\delta_1+1})$, and thus they all belong to the same subbouquet $B_1$, which is either free or non-mixed. In any case, the vector $\cb_{B_1}=({\bf 1},{\bf 0},\ldots,{\bf 0})\in\ZZ^{\delta}$ with ${\bf 1}=(1,\ldots,1)\in\ZZ^{\delta_1+1}$ and $\ab_{B_1}=\ab_1+\cdots+\ab_{\delta_1+1}$. Analogously, we have $\cb_{B_2}=({\bf 0},{\bf 1},\ldots,{\bf 0})$ with ${\bf 1}=(1,\ldots,1)\in\ZZ^{\delta_2+1}$ and $\ab_{B_2}=\ab_{\delta_1+2}+\cdots+\ab_{\delta_1+\delta_2+2}$, and so on, until $\cb_{B_2}=({\bf 0},{\bf 0},\ldots,{\bf 1})$ with ${\bf 1}=(1,\ldots,1)\in\ZZ^{\delta_n+1}$ and $\ab_{B_n}=\ab_{\delta-\delta_n}+\cdots+\ab_{\delta}$. One can easily see from the construction of $A$ that $I_{A_B}=I_D$. In addition, since $I_D\neq 0$ then at least one subbouquet is non-free, and therefore $I_{\MH}=I_A$ is a stable toric ideal. Applying now Theorem~\ref{stable_toric} and \cite[Theorem 3.11]{PTV} we obtain the desired conclusions. \qed       

\end{proof}

\begin{Remark}\label{homological_data}
{\em In conclusion, the previous theorem can be regarded as a polarization-type operation for positively graded toric ideals by comparison to the properties of the classical polarization for monomial ideals, see \cite[Corollary 1.6.3]{HH}. Indeed, if $I_D\subset S=K[x_1,\ldots,x_n]$ and $I_{\MH}\subset T=K[y_1,\ldots,y_{\delta}]$ are the ideals from the above construction since $I_D$ is the subbouquet ideal of $I_{\MH}$ then $\height(I_D)=\height(I_{\MH})$. Applying now \cite[Theorem 3.11]{PTV} we have $\pd(T/I_{\MH})=\pd(S/I_D)$, the $\NN D$-graded Betti numbers of $I_D$ are obtained from the $\NN A_{\MH}$-graded Betti numbers of $I_{\MH}$ and viceversa, and thus $S/{I_D}$ is Cohen-Macaulay (respectively Gorenstein) if and only if $T/I_{\MH}$ is Cohen-Macaulay (respectively Gorenstein).}
\end{Remark}
 
\begin{Remark}\label{robust_conj}
{\em Moreover, combinatorial classification problems of arbitrary positively graded toric ideals whose different sets of bases are equal can be reduced to problems about toric ideals of hypergraphs. For example, in \cite{BBDLMNS} Boocher et.\ al proved that, for robust toric ideals of graphs, the Graver basis is a minimal generating set. In other words, robust toric ideals of graphs are strongly robust. They ask if this property is true in general for \emph{any} robust ideal. To prove such a statement, it is enough to prove it only for toric ideals of hypergraphs. Indeed, if $I_A$ is any robust toric ideal then Theorem~\ref{general_complexity} shows that $I_{\MH}$ is robust. Then if one can prove that robust ideals of hypergraphs are strongly robust then it follows, again from Theorem~\ref{general_complexity}, that $I_A$ is also strongly robust.}
\end{Remark}

\end{document}